\documentclass{amsart}
\usepackage{amsmath,amssymb}
\usepackage{graphicx,subfigure}
\usepackage{color}
\usepackage{faktor}
\usepackage{pb-diagram}

\newtheorem{theorem}{Theorem}[section]
\newtheorem{proposition}[theorem]{Proposition}
\newtheorem{corollary}[theorem]{Corollary}

\newtheorem{lemma}[theorem]{Lemma}

\theoremstyle{definition}

\def\es{{\SS^1_{\ell}}}

\def\tv{\widetilde{v}}

\theoremstyle{remark}
\newtheorem{remark}[theorem]{Remark}

\numberwithin{equation}{section}

\newcommand{\al}{\alpha}
\newcommand{\be}{\beta}
\newcommand{\de}{\delta}
\newcommand{\ep}{\varepsilon}

\newcommand{\ga}{\gamma}
\newcommand{\ka}{\kappa}
\newcommand{\la}{\lambda}
\newcommand{\om}{\omega}
\newcommand{\si}{\sigma}
\newcommand{\te}{\theta}
\newcommand{\vp}{\varphi}

\newcommand{\De}{\Delta}
\newcommand{\Ga}{\Gamma}
\newcommand{\La}{\Lambda}

\def\th{{^{\mathrm{th}}}}

\def\NN{\mathbb{N}}
\def\RR{\mathbb{R}}

\def\ZZ{\mathbb{Z}}

\def\TT{\mathbb{T}}

\renewcommand\SS{\mathbb{S}}
\newcommand{\cA}{{\mathcal A}}

\newcommand{\cE}{{\mathcal E}}

\newcommand{\cI}{{\mathcal I}}

\newcommand{\cN}{{\mathcal N}}
\newcommand{\cO}{{\mathcal O}}
\newcommand{\cP}{{\mathcal P}}

\newcommand{\cR}{{\mathcal R}}
\newcommand{\cS}{{\mathcal S}}
\newcommand{\cT}{{\mathcal T}}

\newcommand{\pd}{\partial}
\newcommand\minus\backslash

\newcommand{\ms}{\mspace{1mu}}

\DeclareMathOperator\Div{div}

\DeclareMathOperator\dist{dist}

\newcommand\DD{\mathbb D}
\renewcommand\leq\leqslant
\renewcommand\geq\geqslant
\newlength{\intwidth}

 \DeclareMathOperator\curl{curl}

\begin{document}

\title[Beltrami fields with hyperbolic orbits and invariant tori]{Beltrami fields with
  hyperbolic periodic orbits \\ enclosed by knotted
  invariant tori}

\author{Alberto Enciso}
\address{Instituto de Ciencias Matem\'aticas, Consejo Superior de
  Investigaciones Cient\'\i ficas, 28049 Madrid, Spain}
\email{aenciso@icmat.es}

\author{Alejandro Luque}
\address{Department of Mathematics, Uppsala University, 751 06
  Uppsala, Sweden}
\email{alejandro.luque@math.uu.se}

\author{Daniel Peralta-Salas}
\address{Instituto de Ciencias Matem\'aticas, Consejo Superior de
  Investigaciones Cient\'\i ficas, 28049 Madrid, Spain}
\email{dperalta@icmat.es}

\date{\today}
%
%
\begin{abstract}
We prove that there exist Beltrami fields in Euclidean space, with 
{sharp decay}
at infinity, which have a
prescribed set of invariant tori (possibly knotted or linked) that
enclose an arbitrarily large number of hyperbolic periodic
orbits. These hyperbolic orbits are cablings over the core curve of
each torus. Moreover, the domain bounded by each invariant torus is covered by an almost full measure set of invariant tori. We show that an analogous result holds for high-frequency Beltrami fields on the flat torus $\TT^3$.
\end{abstract}
\maketitle

\section{Introduction}

Beltrami flows are divergence-free vector fields that satisfy the
equation
\begin{equation}\label{Beltrami}
\curl u=\la u
\end{equation}
in $\RR^3$ with a constant proportionality factor~$\la$. They have long played a fundamental role
in fluid mechanics: although it is classical that they are stationary
solutions of the Euler equations,
\[
\pd_t u+(u\cdot\nabla)u=-\nabla p\,,\qquad \Div u=0\,,
\]
their true relevance was unveiled by Arnold through his celebrated
structure theorem for inviscid fluids in equilibrium. It roughly
asserts that, under mild technical assumptions, a stationary solution of the
Euler equations is either integrable or a Beltrami field. More
precisely, one has the following:

\begin{theorem}[Arnold's structure theorem \cite{Ar65,AK}]
Let $u$ be an analytic stationary solution of the Euler equations in $\RR^3$
that is bounded as
\[
|u(x)|+|\curl u(x)|< C(1+|x|)\,.
\]
If $u\times\curl u$ is not identically zero, $u$ admits an analytic
first integral whose regular level sets are tori, cylinders or planes.
\end{theorem}

Motivated by H\'enon's numerical simulations~\cite{He66}, Arnold
conjectured (\cite[page 19]{Ar65} and \cite[page 347]{Ar66}) that the
hypothesis on $u\times\curl u$ is essential, and that there should be
Beltrami fields whose dynamics has an arbitrarily complicated topology
and the same complexity as a mechanical system with two degrees of
freedom. In the light of the modern theory of Hamiltonian dynamics,
which Arnold himself greatly contributed to establish, it is natural
to interpret this conjecture as the existence of invariant tori of
complicated topology enclosing many homoclinic connections that
intersect transversally. The first part of this assertion, that is, the existence of knotted invariant tori, was
established in~\cite{Acta}. This result can be stated as follows, where structural stability means that the invariant torus is
preserved up to an ambient diffeomorphism under $C^{3,\beta}$~small
divergence-free perturbations of the vector field:

\begin{theorem}[Realization theorem \cite{Acta}]\label{T.Acta}
  Let $\cT_1,\dots,\cT_N$ be a collection of (possibly knotted and
  linked) toroidal domains
  embedded in $\RR^3$ having pairwise disjoint closures. Then for any constant~$\la\neq0$
  there is a diffeomorphism~$\Phi$ of~$\RR^3$ and a Beltrami field
  satisfying $\curl u=\la u$ in $\RR^3$ such that the boundaries $\pd\Phi(\cT_1),\dots, \pd\Phi(\cT_N)$ are
  a collection of invariant tori of~$u$ on which the
  flow is quasiperiodic. These invariant tori are
  structurally stable and the Beltrami field has sharp decay at infinity, i.e. $|u(x)|<\frac{C}{1+|x|}$.
\end{theorem}

The interest of these questions is not merely academic. Invariant tori
of Beltrami fields (or, more generally, of the vorticity $\om:=\curl
u$ associated with a solution to the Euler equations) play a key role
in fluid mechanics, where they are known as vortex tubes. The study of
knotted vortex tubes, which goes back to Lord Kelvin in the
XIX~century~\cite{Kelvin}, is a central topic in the Lagrangian theory
of turbulence and has been extensively pursued in the last decades
(see e.g.~\cite{Kh05,Mo14} for recent accounts of the subject).
Vortex tubes of complicated knotted topologies have been
experimentally constructed in~\cite{Irvine}. The very condition that a
field satisfies the Beltrami equation (that is, that the vorticity be
proportional to the velocity) appears in the study of turbulence, and
in fact experiments and numerical simulations show that in the region
where a fluid presents turbulent behavior, the vorticity and the
velocity tend to align: this phenomenon is usually called
Beltramization~\cite{Farge,Monchaux}. An analysis of the role of
Beltrami fields and Arnold's structure theorem in the context of
laminar and turbulent fluid flows can be found
in~\cite{Yudovich,ARMA}.

Our objective in this paper is to go one step further in order to
establish Arnold's vision of Beltrami fields. More precisely, we show
that there are Beltrami fields with invariant tori of arbitrary
topology that enclose regions with any prescribed number of hyperbolic
periodic orbits. According to Katok's theorem~\cite{Katok}, the
existence of hyperbolic orbits is a necessary condition for having
positive topological entropy. Therefore, our specific goal in this
paper is to show the existence of hyperbolic trajectories enclosed by the
knotted invariant tori constructed in Theorem~\ref{T.Acta}. This is
accomplished in the following theorem. 
Concerning the statement, we recall that a curve in space is said
to be the~\emph{core knot}\/ of a toroidal domain~$\cT$ (which is unique up to an
isotopy) if the domain 
{deform retracts}
onto the curve.

\begin{theorem}\label{T.main}
Take a positive integer~$M$ and any $\de>0$. There exists a Beltrami field as in
Theorem~\ref{T.Acta} that has at least $M$~hyperbolic periodic orbits
in each of the solid invariant tori $\Phi(\cT_1), \dots,
\Phi(\cT_N)$. In each of these toroidal domains, these hyperbolic orbits are
isotopic to each other and cables of the corresponding core
knot. Moreover, these domains are covered
by a set of invariant tori~$\mathcal I_k\subset \Phi(\cT_k)$ of almost
full measure, that is, with $|\mathcal I_k|/|\Phi(\cT_k)|>1-\de$.
\end{theorem}

It is worth stressing that this result is subtler than it
looks. Indeed, at first sight one can be tempted to believe that it
should be quite easy to pass from the many-tori configuration
presented in Theorem~\ref{T.Acta} to the field of
our main theorem using generic perturbations. However, a moment's thought reveals that one
cannot hope to use generic perturbations in our case. The reason is that we are restricting
our attention to the family of Beltrami fields, which is highly
non-generic in itself. This is because a Beltrami field satisfies the
PDE~\eqref{Beltrami}, which in particular means that the family of
Beltrami fields is of infinite codimension in the space of
divergence-free fields. An illustrative example of the (many)
obstructions that this imposes is that a Beltrami field cannot admit a
local first integral with a regular level set diffeomorphic to the
sphere~\cite{ARMA}.

In addition to the fact that the family of Beltrami fields is strongly
non-generic, as we have just discussed, there is another very
important factor that accounts for the difficulty of the problem: to
analyze invariant tori of complicated topologies, in the spirit of
Arnold, one cannot simply use Beltrami fields given by simple explicit
expressions, such as the
{ABC~flow}, 
but one needs to consider general
solutions to a PDE that cannot be solved in closed form. This is a key
difficulty that does not appear in the study of, say, celestial
mechanics or geodesic flows, since in these cases the dynamical
systems are explicitly determined by the masses of the bodies or the
Riemannian metric. To overcome this problem one must combine
dynamical systems techniques with fine estimates for PDEs. Notice
that, if one were interested in simple examples of Beltrami fields
such as 
{the ABC~flow}
on the 3-torus, it is much easier to show that
there can be hyperbolic behavior~\cite{ABC}, but there is no hope of
proving a result like Theorem~\ref{T.main} using only explicit
solutions.

Let us now discuss in some detail the strategy of the proof of the
main theorem. The starting point is the Beltrami
field~$u$ with knotted toroidal invariant domains $\Phi(\cT_k)$ whose existence is granted by
Theorem~\ref{T.Acta}. Indeed, these domains can be taken arbitrarily thin
and their thickness, which we will denote by~$\ep$, plays the role of
a perturbation parameter. One might think that, as there is some
flexibility in the election of parameters leading to this Beltrami
field, it should be possible to exploit it to show the existence of hyperbolic periodic orbits via a subharmonic Melnikov method.
However, this boils down to a
perturbation of order $\ep^3$ of an a priori stable integrable field
with twist of order $\ep^2$. This setting is ideal for the
application of a KAM argument to prove the existence of an almost full measure set of invariant tori
but it is of no use to show the existence of hyperbolic periodic orbits
because the Melnikov subharmonic integral turns out to be degenerate.

This is the key aspect that makes the proof of
Theorem~\ref{T.main} fundamentally different from that of
Theorem~\ref{T.Acta}. To go around it, we will need to come up with
two further carefully concocted perturbations, acting at different
scales, which we extract from the PDE for Beltrami fields and involve perturbing the invariant
tori~$\partial\Phi(\cT_k)$. The first perturbation, which is of order $\ep^2$, is chosen to create $M$ resonant approximate invariant tori with the same $\ep$-independent frequency vectors. The second
perturbation, which is of order~$\ep^{5/2}$, is designed to destroy
these resonant tori and create
hyperbolic periodic orbits in each of the domains
$\Phi(\cT_k)$. To this
end we need to develop a subharmonic Melnikov theorem for Beltrami
fields (Theorem~\ref{T:subh:Melnikov}). Recall that the use of subharmonic
Melnikov integrals to create hyperbolic periodic orbits has found
remarkable recent applications~\cite{GHS12,HKS,KS,RR}.

Let us conclude with a couple of remarks about the framing of this
result in the context of Arnold's conjecture on chaotic vortex
lines. The first observation is that, although hyperbolic vortex lines
of complicated topologies were constructed in~\cite{Annals}, they were
not confined in invariant tori, so this can be considered to be an
unrelated result. The second observation is that, although we have
constructed a wealth of 
{hyperbolic orbits}
inside the invariant
tori, it is very far from trivial to 
{guarantee the existence of transverse homoclinic intersections
that would give rise to the desired chaotic behavior, even by introducing
additional perturbations}. The difficulty is twofold. One the one hand, the problem
turns out to be a priori stable, which makes the expected splitting
exponentially small. {Combined with the intrinsic error in any
  characterization of the Beltrami fields, this fact prevents to apply the techniques available in the literature}.
On the other hand, since the vector field is the
solution to a PDE (and therefore not explicit), it is not clear how to
even compute the field up to exponentially small errors. Again because
of the PDE, the introduction of fast oscillations in the field via
boundary data (which may be a way around some of these problems, as
explored in~\cite{CMP}) presents very nontrivial analytic difficulties
on the PDE side of the problem.

The paper is organized as follows. In Section~\ref{S.Beltramis} we
derive estimates for Beltrami fields on thin tubes that, building on
our previous work~\cite{Acta}, permits to understand Beltrami fields
on thin tubes in terms of suitable boundary data and small
errors. After making a suitable choice of these boundary data, in
Section~\ref{S.trajectories} we analyze the Poincar\'e map of the
corresponding Beltrami field and prove a KAM~theorem for this class of
fields. Armed with these results, in Section~\ref{S.Proof} we prove
our main theorem for Beltrami fields on~$\RR^3$. In
Section~\ref{S.torus} we derive an analogous result
for Beltrami fields on the torus~$\TT^3$ as a consequence of the
result on~$\RR^3$ and of an inverse localization argument. The paper
concludes with an Appendix, both of independent interest and
instrumental for the results of Sections~\ref{S.Beltramis}
and~\ref{S.trajectories}, where we study in detail the dynamics of the
harmonic field on a thin torus.

\section{Beltrami fields on thin toroidal domains}
\label{S.Beltramis}

This section is divided in three parts. In Subsection~\ref{S.Belt1} we
start by introducing a coordinate
system that is well suited to the task of describing functions
and vector fields defined on a thin tube. Next, in Subsection~\ref{sec:harmo}, we provide explicit asymptotic formulas for the (unique) harmonic
field~$h$ in a thin tube of thickness~$\ep$. Finally, in Subsection~\ref{ssec:lBF} we give estimates for Beltrami fields on a thin tube with prescribed harmonic projection and normal component on the boundary of the tube.

\subsection{Coordinates on thin tubes}\label{S.Belt1}
Following~\cite{Acta}, we characterize a
thin tube in terms of the curve that sits on its core and its
thickness $\ep$, which is a parameter that will be everywhere assumed
to be suitably small.

Let us start with a closed smooth curve parametrized by arc-length $\ga:\es\to \RR^3$, with $\es:=\RR/\ell\ms
\ZZ$ (throughout the paper, when the period is $2\pi$ we will simply
write $\SS^1\equiv \SS^1_{2\pi}$). This amounts to saying that the tangent field~$\dot\ga$ has unit norm and $\ell$
is the length of the curve. We will abuse the
notation and denote also by $\ga$ the curve in space defined by the above
map (i.e., the image set $\ga(\es)\subset\RR^3$).

Let us denote
by $\cT_\ep\equiv \cT_\ep(\ga)$ a metric neighborhood with
thickness~$\ep$ of the curve~$\ga$, that is,
\[
\cT_\ep:=\big\{x\in\RR^3:\dist(x,\ga)<\ep\big\}\,.
\]
This is a thin tube having the curve $\ga$ as its core. It is
standard that, for small $\ep$, $\cT_\ep$ is a domain with smooth
boundary.

Since the curvature of a generic curve does not vanish (see
\cite[p.~184]{Bruce}, where ``generic''
refers to an open and dense set, with respect to a reasonable $C^k$
topology, in the space of smooth curves in $\RR^3$), by taking a small
deformation of the curve $\ga$ if necessary one can assume that the
curvature of $\ga$ is strictly positive. This enables us to define the
normal and binormal fields to the curve at each point $\ga(\al)$,
which we will respectively denote by $e_1(\al)$ and $e_2(\al)$.

Using the vector fields $e_j(\al)$ and denoting by $\DD$ the two-dimensional
unit disk, we can introduce smooth coordinates
$(\al,y)\in\es\times\DD$ in the tube $\cT_\ep$ via the diffeomorphism
\[
(\al,y)\mapsto \ga(\al)+\ep\ms y_1\ms e_1(\al)+\ep \ms y_2\ms e_2(\al)\,.
\]
In the coordinates $(\al,y)$, a short computation using the Frenet
formulas shows that the Euclidean metric in the tube reads as
\begin{equation}\label{ds}
ds^2=A\,d\al^2-2\ep^2\tau(y_2\,dy_1-y_1\,dy_2)\,d\al+\ep^2\big(dy_1^2+dy_2^2\big)\,,
\end{equation}
where $\ka\equiv\ka(\al)$ and $\tau\equiv\tau(\al)$ respectively
denote the curvature and torsion of the curve,
\begin{equation}\label{A1}
A:=(1-\ep\ka y_1)^2+(\ep\tau)^2|y|^2\,,
\end{equation}
and $|y|$ stands for the Euclidean norm of $y=(y_1,y_2)$.

We will sometimes take polar coordinates $r\in(0,1)$, $\te\in\SS^1:=\RR/2\pi\ms\ZZ$  in the disk~$\DD$, which are
defined so that
\[
y_1=r\cos\te\,,\quad y_2=r\sin\te\,.
\]
The metric then reads as
\begin{equation}\label{dspolar}
ds^2=A\,d\al^2+2\ep^2\tau
r^2d\te\,d\al+\ep^2dr^2+\ep^2r^2d\te^2\,,
\end{equation}
where we, with a slight abuse of notation, still call $A$ the
expression of~\eqref{A1} in these coordinates, i.e.,
\begin{equation}\label{A2}
A:=(1-\ep\ka r\cos\te)^2+(\ep\tau r)^2\,.
\end{equation}

For future reference, let us record here that the gradient and the
Laplacian of a scalar function on the tube read in these
coordinates as
\begin{align}
\nabla \psi &=
\frac{1}{B^2} \left(\frac{\pd \psi}{\pd \al} - \tau \frac{\pd \psi}{\pd \theta} \right)\pd_\al
+
\frac{1}{\ep^2} \frac{\pd \psi}{\pd r} \pd_r
+
\frac{1}{(\ep r B)^2} \left(A \frac{\pd \psi}{\pd \theta} - \ep^2 r^2 \tau \frac{\pd \psi}{\pd \al} \right) \pd_\theta \,.
\nonumber
\label{eq:nabla}
\\
\De\psi  &= \frac1{\ep^2}\Big( \frac{\partial^2 \psi}{\partial r^2}+\frac1r
\frac{\partial \psi}{\partial r}+\frac
A{r^2B^2} \frac{\partial^2 \psi}{\partial \te^2}\Big)+\frac1{B^2}
\frac{\partial^2
\psi}{\partial \al^2} -
\frac{2\tau}{B^2} \frac{\partial^2 \psi}{ \partial \al \partial \te}\\
&\qquad- \frac{\tau'-\ep
  r(\ka\tau'-\ka'\tau)\cos\te}{B^3} \frac{\partial \psi}{\partial \te} \notag
+
\frac{\ka\sin\te(B^2-(\ep\tau r)^2)}{\ep r B^3} \frac{\partial \psi}{\partial
\te}-\frac{\ka \cos\te}{\ep B} \frac{\partial \psi}{\partial r}\\
&\qquad \qquad\qquad \qquad\qquad \qquad\qquad\qquad\qquad\qquad\quad +\frac{\ep r(\ka'\cos\te+\tau\ka\sin\te)}{B^3} \frac{\partial \psi}{\partial
\al} \,,
\label{eq:laplacian}
\end{align}
with $B:=1-\ep\ka r\cos\te$, and that the Euclidean volume measure is
written as $dx= \ep^2 B\, d\al\, dy$. Here and in what follows we denote
derivatives with respect to the variable~$\al$ by primes. We will set
\begin{equation}
\label{eq:laplacian:y}
\De_y \psi = \frac{\partial^2 \psi}{\partial r^2}+\frac1r
\frac{\partial \psi}{\partial r}+\frac
1{r^2} \frac{\partial^2 \psi}{\partial \te^2}\,.
\end{equation}

Finally, we shall denote by $\cO(\ep^j)$ any quantity $q(\al,y)$
defined on $\es\times\DD$ such that
\[
\|q\|_{C^k(\es\times\DD)}\leq C\ep^j\,,
\]
where the constant $C$ depends on~$k$ but not on~$\ep$.

\begin{remark}
In this paper we follow the most popular sign convention for the torsion of a curve, which is the opposite as the one considered in~\cite{Acta}. This explains the difference between some formulas obtained here and those obtained in~\cite{Acta} (which become the same after the transformation $\tau \mapsto -\tau$).
\end{remark}

\subsection{Asymptotic formulas for harmonic fields on thin tubes}\label{sec:harmo}

The (tangent) harmonic fields on the
tube $\cT_\ep$, are defined as the vector fields $h \in
C^\infty(\cT_\ep,\RR^3)$ such that
\[
\curl h=0\,,\qquad \Div h=0\,, \qquad h\cdot \nu=0\,,
\]
where $\nu$ is a unit normal. By Hodge theory, harmonic fields on a toroidal domain define a one-dimensional linear
space, so in what follows we shall fix, once and for all, a nonzero
harmonic field that we will still call $h$. Our goal in this subsection is to provide explicit formulas for the harmonic
field~$h$ and its derivatives up to terms that are suitably small for
small~$\ep$.

Since the vector field
\[
h_0 := B^{-2} (\pd_\al - \tau \pd_\theta)
\]
can be readily shown to be irrotational and tangent to the boundary,
by the Hodge decomposition it follows that the harmonic field can be
written as
\begin{equation}\label{harm}
h=:h_0+\nabla\vp\,,
\end{equation}
where
\begin{equation}\label{eq:New:BVP:harmonic}
  \Delta \vp = \rho_0 \quad \mbox{in }  \cT_\ep, \qquad \frac{\pd \vp}{\pd \nu} =0
  \,,\qquad \int_{\cT_\ep}\vp\, dx=0\,,
\end{equation}
and
\[
\rho_0 :=-\Div h_0= -\ep B^{-3} r (\tau \kappa \sin\theta + \kappa' \cos \theta) \,.
\]

The function $\vp$ that determines the harmonic field $h$ can be computed perturbatively as an expansion in $\ep$. To this end, we recall the estimates obtained in~\cite[Theorem 4.9]{Acta} for the Neumann boundary value problem
\begin{align}\label{Neumann1}
\Delta \psi = \rho \quad \text{in } \cT_\ep\,,
\qquad \frac{\pd \psi}{\pd\nu} =0\,,\qquad \int_{\cT_\ep}\psi\, dx=0\,.
\end{align}
To state quantitative estimates for a scalar function
$\psi$ on $\cT_\ep$, we will use the coordinates $(\al,y)$ to define the $k\th$ Sobolev norm as
\[
\|\psi\|_{H^k}^2:=\sum_{i_1+i_2+i_3\leq k} \int_{\es\times\DD}
\bigg|\frac{\pd^{i_1+i_2+i_3} \psi}{\pd^{i_1}y_1\, \pd^{i_2}y_2\, \pd^{i_3}\al}\bigg|^2\, d\al\, dy\,.
\]

\begin{theorem}[\cite{Acta}]\label{T.estimates1}
For small enough~$\ep$, the boundary value problem~\eqref{Neumann1} has a
unique solution $\psi$, provided that $\rho$ satisfies the necessary
condition
\[
\int_{\cT_\ep}\rho\, dx=0\,.
\]
For any integer $k$, the solution is
  bounded as
\begin{align*}
\|\psi\|_{H^k}&\leq C\,\|\rho\|_{H^k}\,,\\
\|D_y\psi\|_{H^k}&\leq C\,\ep^2\|\rho\|_{H^{k+1}}\,.
\end{align*}
The constants depend on~$k$ but not on $\ep$.
\end{theorem}

Perturbative computations and the estimates in this theorem allow us
to obtain the explicit expansion of the harmonic field $h$.

\begin{proposition}\label{prop:HF}
The components of the harmonic field, defined as
\[
h =: h_\al \pd_\al + h_r \pd_r + h_\te \pd_\te\,,
\]
can be written as
\begin{equation}\label{eq:HF:unscaled}
\begin{split}
h_\al &= 1 + \ep h_\al^{(1)}+
\ep^2 h_\al^{(2)} + \cO(\ep^3) \, ,\\
h_r&= \ep h_r^{(1)}+
\ep^2 h_r^{(2)} + \cO(\ep^3) \, , \\
h_\te&= - \tau + \ep h_\theta^{(1)} + \ep^2 h_\theta^{(2)} + \cO(\ep^3)\, , \\
\end{split}
\end{equation}
where
\begin{align*}
h_\al^{(1)} &:= 2 \kappa r \cos \theta\, , \\
h_r^{(1)} &:= -\frac{3(r^2 - 1) }{8} (\tau \kappa \sin \theta + \kappa' \cos \theta) \, , \\
h_\theta^{(1)} &:= -2 \tau \kappa r \cos \theta + \frac{r^2-3}{8r} (-\tau \kappa \cos \theta + \kappa' \sin \theta) \, , \\
h_\al^{(2)} &:= 3 \kappa^2 r^2 \cos^2 \theta+H_2(\al,r)\, , \\
h_r^{(2)} &:= -\frac{13 (r^3- r) }{24} (\tau \kappa^2 \sin {2\theta}+ \kappa \kappa' \cos {2\theta}) +  H_1(\al,r) , \\
h_\theta^{(2)} &:= -3 \tau \kappa^2 r^2 \cos^2 \theta + \frac{13(r^2 -2)}{48}  (-\tau \kappa^2 \cos 2\theta+\kappa \kappa' \sin 2\theta) - \tau H_2(\al,r)\, ,
\end{align*}
where the smooth functions $H_1(\al,r)$ and $H_2(\al,r)$ are
$\ell$-periodic in $\al$ with zero mean (that is, $\int_0^\ell
H_j(\al,r)\, d\al=0$) and independent of $\ep$.
\end{proposition}
\begin{proof}
Following~\cite[Theorem
5.1]{Acta}, the function $\vp$ can be written as
\[
\vp = \vp_0 + \vp_1 + \vp_2 + \vp_3
\]
where
\begin{align*}
\varphi_0 &:=  -\frac{r^3 - 3 r}{8}\ep^3 (\tau \kappa \sin \theta + \kappa' \cos \theta) \, , \\
\varphi_1 &:=  -\frac{13(r^4-2 r^2)}{96}\ep^4  (\tau \kappa^2 \sin {2\theta}+ \kappa \kappa' \cos {2\theta}) \,,
\end{align*}
and the functions $\vp_2$ and~$\vp_3$ are respectively defined as the unique solutions of
{the boundary value problem}
\[
\frac{\pd^2 \vp_2}{\pd\al^2} + \frac{\De_y\vp_2}{\ep^2} = \frac{3-14 r^2}{8} \ep^2  \kappa \kappa' \mbox{ in } \es\times\DD,
\quad \frac{\pd \vp_2}{\pd r}\Big|_{r=1} =0
\,,
\quad
\int_{\es\times\DD} \vp_2 \, d \al \,d y=0\,,
\]
where $\De_y\vp_2:= \frac{\pd^2 \vp_2}{\pd r^2} + \frac1r\frac{\pd
  \vp_2}{\pd r} + \frac{\pd^2 \vp_2}{\pd \te^2}$ is the standard
Laplacian in the $y$-coordinates (written here in polar coordinates), and
\[
\Delta \vp_3 = \cO(\ep^3) \mbox{ in } \es\times\DD,
\quad \frac{\pd \vp_3}{\pd r}\Big|_{r=1} =0
\,,
\quad
\int_{\es\times\DD} \vp_3 \, d \al\, d y=0\,.
\]
Using the above expressions together with the estimates in Theorem~\ref{T.estimates1}
with $\rho=\rho_0$, we
obtain
\begingroup
\allowdisplaybreaks
\begin{align*}
\frac{\pd \vp}{\pd\al} = {} &
 \underbrace{\frac{\pd \vp_0}{\pd\al}}_{\cO(\ep^3)}
 +
 \underbrace{\frac{\pd \vp_1}{\pd\al}}_{\cO(\ep^4)}
+
\underbrace{\frac{\pd \vp_2}{\pd\al}}_{\cO(\ep^2)}
+
\underbrace{\frac{\pd \vp_3}{\pd\al}}_{\cO(\ep^3)}\,,\\
\frac{\pd \vp}{\pd r} = {} &
\underbrace{\frac{\pd \vp_0}{\pd r} }_{\cO(\ep^3)}
+
\underbrace{\frac{\pd \vp_1}{\pd r} }_{\cO(\ep^4)}
+
\underbrace{\frac{\pd \vp_2}{\pd r} }_{\cO(\ep^4)}
+
\underbrace{\frac{\pd \vp_3}{\pd r} }_{\cO(\ep^5)}\,,\\
\frac{\pd \vp}{\pd \te} = {} &
\underbrace{\frac{\pd \vp_0}{\pd \te} }_{\cO(\ep^3)}
+
\underbrace{\frac{\pd \vp_1}{\pd \te} }_{\cO(\ep^4)}
+
\underbrace{\frac{\pd \vp_2}{\pd \te} }_{=0}
+
\underbrace{\frac{\pd \vp_3}{\pd \te} }_{\cO(\ep^5)}\,.
\end{align*}
\endgroup
By gathering the terms with the same dependence on~$\ep$, this chart
can be readily used to compute $h$
perturbatively. Indeed, using now the explicit formulas
$A:=(1-\ep\ka r\cos\te)^2+(\ep\tau r)^2$ and
$B:=1-\ep\ka r\cos\te$ together with the formula of the gradient in
the coordinates $(\al,r,\te)$, we can compute the components of
the harmonic field~$h$  (Equation~\eqref{harm})
as follows:
\begin{align*}
h_\al&=B^{-2}\bigg(1+\frac{\pd \vp}{\pd\al}-\tau  \frac{\pd
       \vp}{\pd\te}\bigg)  \\
 &=
1+ 2\ep \kappa r \cos \theta  + \ep^2 3 \kappa^2 r^2 \cos^2 \theta + \pd_\al \vp_2 + \cO(\ep^3)\,,
\\
h_r&=\ep^{-2}\frac{\pd \vp}{\pd r}\\
&=
\ep^{-2}\bigg(\frac{\pd \vp_0}{\pd r}+ \frac{\pd \vp_1}{\pd r}+ \frac{\pd \vp_2}{\pd r}\bigg)+ \cO(\ep^3)\,,
\\
h_\te&=B^{-2}\bigg(-\tau + \frac A{\ep^2 r^2}  \frac{\pd \vp}{\pd \te}-\tau
       \frac{\pd \vp}{\pd \al}\bigg) \\
&=
-\tau - 2 \ep \tau \kappa r \cos \theta + \frac{\frac{\pd \vp_0}{\pd \te}}{\ep^2 r^2}
 - 3 \ep^2 \tau \kappa^2 r^2 \cos^2 \theta - \tau \frac{\pd \vp_2}{\pd \al} + \frac{\frac{\pd \vp_1}{\pd \te}}{\ep^2 r^2} + \cO(\ep^3)
\end{align*}
The formula in the statement is obtained upon substituting the
formulas for $\vp_0$ and $\vp_1$ and setting
\[
H_1(\al,r):=\Big(\frac{1}{\ep^4} \frac{\pd \vp_2}{\pd
  r}\Big)\Big|_{\ep=0}\,,\qquad H_2(\al,r):=\Big(\frac{1}{\ep^2} \frac{\pd \vp_2}{\pd
  \al}\Big)\Big|_{\ep=0}\,.
\]
Here we are using that $\vp_2$ is independent of $\te$ because the
problem that defines $\vp_2$ is invariant under rotations of the coordinate~$\te$. It is obvious that $H_2$ has zero mean in $\alpha$. To see that the same holds for $H_1$, it is enough to integrate the elliptic PDE defining $\vp_2$ with respect to the $\alpha$-variable to check that
$$
\frac{\partial}{\partial r}\Big(r\int_0^\ell H_1(\alpha,r)\, d\al\Big)=0\,.
$$
Since the function $\vp_2$ is smooth, this readily implies that $H_1$ has zero mean as well, and the proposition follows.
\end{proof}

\begin{remark}
With some more work, one can show that
\[
H_1(\al,r)=-\frac7{16}r\,\ka(\al)\,\ka'(\al)\,,
\]
but we have not been able to derive a simple expression for
$H_2(\al,r)$. Anyhow, their explicit expressions will not be important
in further sections due to some unexpected cancellations that will
appear later on.
\end{remark}

\subsection{Estimates for Beltrami fields with prescribed normal component}
\label{ssec:lBF}

We will need to control Beltrami fields with prescribed projection
on the space of harmonic fields and prescribed normal component. That is, we need estimates (with the
sharp dependence of the thickness~$\ep$) for the
system of PDE
\begin{align}\label{curlv}
\curl v=\la v \quad \text{in } \cT_\ep\,,\qquad v\cdot
\nu=f\,,\qquad \cP_h v=1\,,
\end{align}
where $f$ can be regarded as a scalar function defined on the boundary
$\SS^1_\ell\times\pd \DD$ (which we identify with
$\SS^1_\ell\times\SS^1$) and the linear functional $\cP_h$ is defined
in terms of the harmonic field~$h$ introduced in the previous
subsection as
\[
\cP_h v:=\frac{\int_{\cT_\ep} v\cdot h\, dx}{\int_{\cT_\ep} |h|^2\, dx}\,.
\]
Since for technical reasons we will be interested in small values of~$\la$, we assume
throughout that the parameter~$\la$ is bounded as $|\la|\leq 1$ (in fact, we could have taken $|\la| \leq
c/\ep$, with $c$ an explicit constant, but we will not need this
refinement).

We shall next see that the analysis of how the boundary datum determines the
Beltrami field $v$ through the equation~\eqref{curlv} ultimately boils
down to the study of the auxiliary scalar boundary value problem~\eqref{Neumann1}.
The norm of a vector field on $\cT_\ep$ is defined componentwise, with the components of a
vector field $w$ on $\cT_\ep$ being
\[
(w_\al,w_y)\equiv (w_\al,w_{y_1},w_{y_2})\,,
\]
where the coordinate expression of $w$ is
\[
w= w_\al\,\pd_{\al}+w_{y_1}\,\pd_{y_1}+w_{y_2}\,\pd_{y_2}\,.
\]

The Beltrami field satisfying~\eqref{curlv} can be written as
\[
v=h+\nabla \psi + \cE\,,
\]
where the function $\psi$ is the only solution to the Neumann
problem:
\begin{align}\label{Neumann2}
\Delta \psi =0 \quad \text{in } \cT_\ep\,,
\qquad \frac{\pd \psi}{\pd\nu} =f\,,\qquad \int_{\cT_\ep}\psi\, dx=0\,.
\end{align}
Accordingly, the field $\cE$ is divergence-free, tangent to the boundary $\partial \cT_\ep$, has zero harmonic projection, and satisfies the equation
\[
(\curl-\la)\cE=\la(\nabla\psi+h)\,.
\]

The following result is a straightforward consequence of~\cite[Proposition 6.7]{Acta} and~\cite[Theorem 6.8]{Acta}
\begin{theorem}[\cite{Acta}]\label{T.estimates}
For small enough~$\ep$, the boundary value problem~\eqref{curlv} has a
unique solution, provided that $f$ satisfies the necessary
condition
\[
\int_{\pd\cT_\ep}f\, d\si=0\,.
\]
Assuming that the unique solution to the Neumann problem~\eqref{Neumann2} satisfies that $\nabla \psi=\cO(1)$ (understood componentwise), for any integer $k$ the components of the field~$\cE$ are bounded as
\[
\ep^{-1}\|\cE_\al\|_{H^k}+ \|\cE_y\|_{H^k}\leq C|\la|\,.
\]
The constant depends on~$k$ but not on $\ep$.
\end{theorem}

In this paper we will be interested in taking $\lambda=\cO(\ep^3)$ and normal data that are linear combinations of functions of the form
\begin{equation}\label{eq:datum}
f(\al,\te)=\ep^s a(\al) e^{i n \te}\,,
\end{equation}
where $n$ is an integer with $|n|\neq 0,1$, $a$ is a smooth
function in $\SS^1_\ell$, and exponents $s\geq 3$. In particular, this function $f$ satisfies the necessary condition in Theorem~\ref{T.estimates}.
Using the adapted coordinates $(\al,r,\te)$, the problem~\eqref{Neumann2} reads as
\begin{equation}\label{eq:Neumann:new}
\Delta \psi = 0 \quad \text{in } \cT_\ep\,,
\qquad \left.\frac{\pd \psi}{\pd r}\right|_{r=1} =f \,,\qquad \int_{\cT_\ep}\psi\, dx=0\,.
\end{equation}

In order to solve this equation perturbatively in $\ep$, we introduce the auxiliary function
\[
\psi_0 = \frac{\ep^s}{n} a(\al) r^n e^{i n \te}
\]
that clearly satisfies $\Delta_y\psi_0=0$ in $\cT_\ep$ and $\frac{\pd \psi_0}{\pd r}|_{r=1} =f$. Additionally it has zero mean $\int_{\cT_\ep}\psi_0\, dx=0$. Hence,
taking $\tilde \psi := \psi-\psi_0$,
the problem~\eqref{eq:Neumann:new} is reduced to
\begin{equation}\label{eq:Neumann:new2}
\Delta \tilde \psi = - \Delta \psi_0
\quad \text{in } \cT_\ep\,,
\qquad \pd_r \tilde \psi|_{r=1} = 0 \,, \qquad \int_{\cT_\ep}\tilde\psi\, dx=0\,.
\end{equation}

Using the expressions~\eqref{eq:laplacian} and~\eqref{eq:laplacian:y}, we can compute the
dominant terms of the Laplacian of $\psi_0$:
\begin{align*}
\De \psi_0 = {} & \frac1{\ep^2} \De_y \psi_0 + \frac{\kappa \sin \te}{\ep r} \frac{\pd \psi_0}{\pd \te}
- \frac{\kappa \cos \te}{\ep} \frac{\pd \psi_0}{\pd r} + \cO(\ep^s)
\\
= {} &
- \ep^{s-1} a \ka r^{n-1} e^{i (n-1) \te} + \cO(\ep^s)\,.
\end{align*}

To obtain the dominant term of the solution of~\eqref{eq:Neumann:new2}, we define $\tilde{\tilde\psi}:=\tilde\psi-\tilde\psi_0$, where $\tilde \psi_0 := Ae^{i (n-1)\te}+C$,
with
\[
A = c_1 r^{n-1} + c_2 r^{n+1}\,,
\]
and $C$ is a constant that is fixed later. Choosing the functions
\[
c_2 := \frac{\ep^{s+1} a \ka}{4n}\,, \qquad c_1 := -c_2 \frac{n+1}{n-1}\,,
\]
we directly check that $\tilde\psi_0$ satisfies the equation
\[
\frac1{\ep^2} \De_y \tilde \psi_0
= \ep^{s-1} a \ka r^{n-1} e^{i (n-1) \te}\,,
\]
the boundary condition $\pd_r \tilde \psi_0|_{r=1} = 0$, and has zero mean $\int_{\cT_\ep}\tilde\psi_0\, dx=0$ for an appropriate choice of the constant $C=\cO(\ep^{s+2})$.

An easy computation shows that $\tilde{\tilde\psi}$ is the unique solution to a Neumann problem of the form
\begin{equation}
\Delta \tilde {\tilde \psi} = \cO(\ep^s)
\quad \text{in } \cT_\ep\,,
\qquad \pd_r \tilde {\tilde \psi}|_{r=1} = 0 \,, \qquad \int_{\cT_\ep}\tilde{\tilde\psi}\, dx=0\,.
\end{equation}
Theorem~\ref{T.estimates1} then implies that $\tilde{\tilde\psi}=\cO(\ep^s)$ and $D_y\tilde{\tilde\psi}=\cO(\ep^{s+2})$.

Putting together the above computations, we conclude that the solution $\psi = \psi_0+\tilde\psi_0+\tilde{\tilde\psi}$
to the Neumann problem~\eqref{eq:Neumann:new} satisfies
\begin{align*}
\frac{\pd \psi}{\pd r} = {} & \ep^s a r^{n-1} e^{i n\te}
+ \ep^{s+1} \frac{a \ka(n+1)}{4 n} \Big( r^n - r^{n-2} \Big) e^{i(n-1)\te} + \cO(\ep^{s+2}) \,, \\
\frac{\pd \psi}{\pd \te} = {} & \ep^s i a r^{n} e^{i n \te}
+ \ep^{s+1} \frac{i a \ka}{4 n} \Big( (n-1) r^{n+1} - (n+1) r^{n-1} \Big) e^{i(n-1)\te} + \cO(\ep^{s+2}) \,.
\end{align*}

Finally, using the expression~\eqref{eq:nabla} we obtain that the components of $\nabla \psi$ are
given by
\begin{align*}
(\nabla \psi)_\al = {} & \cO(\ep^s) \,,\\
(\nabla \psi)_r = {} & \ep^{s-2} a r^{n-1} e^{i n\te}
+ \ep^{s-1} \frac{a \ka (n+1)}{4 n} \Big(r^n - r^{n-2} \Big) e^{i(n-1)\te} + \cO(\ep^s) \,, \\
(\nabla \psi)_\te = {} &
 \ep^{s-2} i a r^{n-2} e^{i n \te}
+ \ep^{s-1} \frac{i a \ka}{4 n} \Big( (n-1) r^{n-1} - (n+1) r^{n-3} \Big) e^{i(n-1)\te} + \cO(\ep^s) \,.
\end{align*}
In particular, since $s\geq 3$, then $\nabla \psi=\cO(\ep)$, so the condition for the function $\psi$ in Theorem~\ref{T.estimates} is satisfied. This will be exploited in the next section to analyze the integral curves of the Beltrami field $v$.

\begin{remark}\label{eq:explicit:n1}
In the above computations we have avoided the case $n=1$ in order to compute
an explicit expression for the terms $\cO(\ep^{s-1})$ of the field $\nabla \psi$.
A weaker estimate in the case $n=1$ will also be useful later. By adapting the
argument to solve the Neumann problem~\eqref{eq:Neumann:new2} one can readily see that for $n=1$ the following
weaker estimate holds:
\begin{align*}
(\nabla \psi)_\al = {} & \cO(\ep^{s})\,, \\
(\nabla \psi)_r = {} & \ep^{s-2} a e^{i\te} + \cO(\ep^{s-1})\,,\\
(\nabla \psi)_\te = {} & \ep^{s-2} i a r^{-1} e^{i\te} + \cO(\ep^{s-1})\,.
\end{align*}
\end{remark}

\section{The Poincar\'e map of the Beltrami field: KAM and resonances}
\label{S.trajectories}

This section consists of four parts. In Subsection~\ref{S.normal} we
consider a Beltrami field arising from normal data that have a certain
structural form. Its Poincar\'e map is computed in
Subsection~\ref{S:Poinc} and analyzed in depth in
Subsection~\ref{S:KAM:case2}, where we make a concrete choice for the
boundary data. To conclude, in Subsection~\ref{S:KAM:case1} we prove a
KAM theorem for this class of Beltrami fields.

\subsection{Normal data of size $\cO(\ep^3)$}\label{S.normal}
In this subsection we will consider the Beltrami field
$v$ on the tube~$\cT_\ep$ constructed in Subsection~\ref{ssec:lBF}, c.f. Equation~\eqref{curlv}. More precisely, we will fix an eigenvalue $\lambda=\cO(\ep^3)$ and
take the following linear combination of functions for the normal component:
\[
f = \sum_{n \in \cN} f_n + \hat f\,,
\]
with
$\cN \subset \NN \cap [2,N]$ a finite set. The smooth (real-valued) functions $f_n$ are taken as
\[
f_n =
\ep^3(a_n e^{i n \te} + \overline{a_n} e^{-i n \te})\,,
\]
where the (complex-valued) functions $a_n\equiv a_n(\alpha)$ will be fixed later, and the function $\hat f$ is taken as
\[
\hat f = \cO(\ep^{2+\mu})\,.
\]
Here and in what follows $2<\mu<3$ is a fixed exponent. At this moment we do not fix the function $\hat f$ because all the analysis in this section is independent of it; we shall choose a convenient function in Section~\ref{S.Proof} in order to create hyperbolic periodic orbits via a subharmonic Melnikov method.

Hence, using the construction and estimates in Subsection~\ref{ssec:lBF}, we produce
a Beltrami field of the form
$$v=h + \nabla \psi + \nabla \hat\psi+ \cO(\ep^3)\,,$$
smooth up to the boundary of the tube. The function $\psi$ is the unique solution to the boundary problem~\eqref{eq:Neumann:new} with Neumann datum given by $\sum_{n \in \cN} f_n$, and $\hat\psi$ is the unique solution to the boundary problem~\eqref{eq:Neumann:new} with Neumann datum $\hat f$.

Using the expressions obtained at the end of Subsection~\ref{ssec:lBF} and the linearity of the Neumann problem, we can write the components of the vector field $\nabla\psi$:
\begingroup
\allowdisplaybreaks
\begin{align}
(\nabla \psi)_\alpha = {} & \cO(\ep^3) \,,
\nonumber \\
(\nabla\psi)_r   = {} & \ep \sum_{n \in \cN} a_n r^{n-1} e^{i n \te}
              + \ep \sum_{n \in \cN} \overline{a_n} r^{n-1} e^{-i n \te}
\nonumber \\
              & + \ep^2 \sum_{n \in \cN} \frac{a_n \ka(n+1)}{4 n}
              \left(r^n - r^{n-2} \right) e^{i (n-1) \te}
\nonumber \\
              & + \ep^2 \sum_{n \in \cN} \frac{\overline{a_n} \ka (n+1) }{4 n}
              \left(r^n - r^{n-2} \right) e^{-i (n-1) \te}
              + \cO(\ep^3)
              \,,
\nonumber \\
         =: {} & \ep u_r^{(1)} + \ep^2 u_r^{(2)} + \cO(\ep^3)
              \,,
\label{eq:ur12} \\
(\nabla\psi)_\te=
          {} & \ep \sum_{n \in \cN} i a_n r^{n-2} e^{i n \te}
               - \ep \sum_{n \in \cN} i \overline{a_n} r^{n-2} e^{-i n \te}
\nonumber \\
              & + \ep^2 \sum_{n \in \cN} \frac{i a_n \ka}{4 n}
              \left( (n-1) r^{n-1} - (n+1) r^{n-3} \right) e^{i (n-1) \te}
\nonumber \\
              &- \ep^2 \sum_{n \in \cN} \frac{i \overline{a_n} \ka}{4 n}
              \left( (n-1) r^{n-1} - (n+1) r^{n-3} \right) e^{-i (n-1) \te}
              + \cO(\ep^3)\,,
\nonumber \\
         =: {} & \ep u_\te^{(1)} + \ep^2 u_\te^{(2)} + \cO(\ep^3)
              \,,
\label{eq:ute12}
\end{align}
\endgroup
and we recall that $a_n\equiv a_n(\al)$, $\ka\equiv\ka(\al)$ and $\tau\equiv\tau(\al)$.

Proceeding in the same way, we obtain that the field $\nabla\hat\psi$ satisfies the estimates
\[
(\nabla\hat\psi)_\alpha = \cO(\ep^3) \,,
\quad
(\nabla\hat\psi)_r = \ep^\mu v_r^{(\mu)} + \cO(\ep^3) \,,
\quad
(\nabla\hat\psi)_\theta = \ep^\mu v_\te^{(\mu)} + \cO(\ep^3) \,,
\]
where the $\ep$-independent dominant terms $ v_r^{(\mu)}, v_\te^{(\mu)}$ depend on the particular choice of the function $\hat f$.

In the following subsections we shall give sufficient conditions on the functions $a_n(\al)$ to characterize the Poincar\'e
map of the above Beltrami field $v$, finding a suitable balance between the application
of KAM theory in the interior of the tube and the existence of
hyperbolic periodic orbits. This task is nontrivial due to the fact that Beltrami fields are non-generic.

\subsection{The Poincar\'e map}\label{S:Poinc}
In order to integrate the vector field, it is convenient to
consider instead of~$v$ the rescaled field
\[
X:=\frac v{v_\al}\,,
\]
which has the same integral curves of~$v$ up to a
reparametrization. Notice that $v_\al$ does not vanish, for small
enough~$\ep$, because $v_\al=h_\al+\cO(\ep^3)=1+\cO(\ep)$ by the asymptotic formulas
for the harmonic field given in~\eqref{eq:HF:unscaled}.

\begin{proposition}\label{prop:BeltramiX}
The components of the vector field~$X$ are given by
\begin{equation}\label{eq:BeltramiX}
\begin{split}
X_\al&= 1+\cO(\ep^3) \, ,\\[1mm]
X_r&  = \ep X_r^{(1)}+
\ep^2 X_r^{(2)} + \ep^\mu X^{(\mu)}_r+ \cO(\ep^3) \, , \\[1mm]
X_\te& =  -\tau + \ep X_\theta^{(1)} + \ep^2 X_\theta^{(2)}  + \ep^\mu X^{(\mu)}_\te+ \cO(\ep^3)\, ,
\end{split}
\end{equation}
where
\begin{align*}
X_r^{(1)} &:= -\frac{3 r^2 -3 }{8} (\tau \kappa \sin \theta+ \kappa'
\cos\theta) + u_r^{(1)} \,, \\
X_r^{(2)} &:= -\frac{r^3-r}{6}(\tau \kappa^2 \sin{2\theta} +
\kappa \kappa' \cos {2\theta})+ \frac{3(r^3-r)}{8} \kappa \kappa' + H_1 \\
& \qquad - 2\ka r \cos{\te} u_r^{(1)} + u_r^{(2)} \,, \\
X_r^{(\mu)} &:= v_r^{(\mu)} \,,\\
 X_\theta^{(1)} &:= \frac{r^2-3}{8r} (-\tau \kappa \cos \theta+\kappa' \sin
\theta) + u_\te^{(1)} \,, \\
 X_\theta^{(2)} &:= \frac{7r^2-8}{48} (-\tau \kappa^2 \cos 2\theta + \kappa
\kappa' \sin {2\theta}) -\frac{3-r^2}{8} \tau \kappa^2 \\
& \qquad - 2\ka r \cos{\te} u_\te^{(1)} + u_\te^{(2)} \,,\\
X_\theta^{(\mu)} &:= v_\theta^{(\mu)}+\tau v_\al^{(\mu)} \, .
\end{align*}
\end{proposition}
\begin{proof}
From the definition of the field~$v$ and the expansion of the harmonic
field we have
\begin{align*}
v_\al = {} & 1 + \ep h_\al^{(1)} + \ep^2 h_\al^{(2)} + \ep^\mu v_\al^{(\mu)} + \cO(\ep^3) \,, \\
v_r = {} &  \ep h_r^{(1)} + \ep u_r^{(1)} + \ep^2 h_r^{(2)} + \ep^2 u_r^{(2)} + \ep^\mu v_r^{(\mu)} + \cO(\ep^3) \,, \\
v_\theta = {} & -\tau+ \ep h_\theta^{(1)} + \ep u_\theta^{(1)} + \ep^2 h_\theta^{(2)} + \ep^2 u_\theta^{(2)}+ \ep^\mu v_\theta^{(\mu)} + \cO(\ep^3) \,.
\end{align*}
Hence one arrives at
\begin{align*}
\frac{v_r}{v_\al} = {} & \ep h_r^{(1)} + \ep u_r^{(1)} +\ep^2 [-h_r^{(1)} h_\al^{(1)}-u_r^{(1)} h_\al^{(1)}+h_r^{(2)}+u_r^{(2)}] + \ep^\mu v_r^{(\mu)} + \cO(\ep^3)\,, \\
\frac{v_\theta}{v_\al} = {} & -\tau + \ep [\tau h_\al^{(1)}+h_\theta^{(1)}+u_\theta^{(1)}] \\
& + \ep^2 [-\tau (h_\al^{(1)})^2 +\tau h_\al^{(2)}-h_\theta^{(1)} h_\al^{(1)}+h_\theta^{(2)}
-u_\theta^{(1)} h_\al^{(1)}+u_\theta^{(2)}] \\
& + \ep^\mu [v_\theta^{(\mu)}+\tau v_\al^{(\mu)}] + \cO(\ep^3)\,.
\end{align*}
The desired expressions are now obtained by substituting the expressions
for the components of the field $h$
derived in Proposition~\ref{prop:HF}.
For convenience, we do not substitute the expressions of
$u_r^{(1)}$,
$u_r^{(2)}$,
$u_\te^{(1)}$, and
$u_\te^{(2)}$, given by~\eqref{eq:ur12} and \eqref{eq:ute12}.
\end{proof}

\begin{remark}
Notice that when $f=0$ (i.e. the Beltrami field has no normal component on the boundary of the tube), we recover the
local Beltrami field considered in~\cite{Acta}.
\end{remark}

The trajectories of $X$ are given by the
parametrization $(\al(s),r(s),\te(s))$ satisfying
\[
\dot \al =1\,,
\qquad \dot r = X_r(\al,r,\te)\,,
\qquad \dot \te = X_\te(\al,r,\te)\,,
\]
with initial condition $(\al(0),r(0),\te(0))=(\al_0,r_0,\te_0)$.
We will also denote by $\phi_s$ the time-$s$ flow of the field $X$, which is a well defined diffeomorphism of $\SS_\ell^1 \times \DD$ for all
values of $s$. Let us now consider the Poincar\'e map of the field $X$, which is the
tool we will use to analyze the dynamical properties of the flow (and which coincides
with that of the local Beltrami field $v$). For this, we start by considering the
section $\{\al=0\}$, which is clearly transverse to the vector field $X$. The
Poincar\'e map of this section, $\Pi : \DD_{R} \rightarrow \DD$ (where
$\DD_R$ denotes the disk of radius $R<1$), sends each point
$(r_0,\te_0) \in \DD_R$ to the first point at which the trajectory
$\phi_s(0,r_0,\te_0)$ intersects the section $\{\al=0\}$ (with $s>0$). The
reason why we are considering the field $X$ is that it is isochronous in the sense
that this first return point is given by the time-$\ell$ flow of $X$, that is,
\begin{equation}\label{eq:Poin:map}
\Pi(r_0,\te_0) := \phi_\ell (0,r_0,\te_0) = (r(\ell),\te(\ell)).
\end{equation}
It should be noticed that~$\Pi$ is area-preserving:

\begin{proposition}[Proposition 7.3 in \cite{Acta}]\label{P:area}
The Poincar\'e map $\Pi:\DD_R\to\DD$ preserves the positive measure
\[
B v_\al\big|_{\al=0}\,r\, dr\, d\te= \big[1+\cO(\ep)\big]\, r\, dr\, d\te\,.
\]
\end{proposition}

Our goal is to write this Poincar\'e map as an asymptotic expansion in $\ep$, so that the existence of
invariant quasiperiodic curves can be proved by applying the KAM theorem, while the existence of hyperbolic periodic points be obtained via a Melnikov subharmonic method. To this end, we will need to select suitable functions $a_n$ in the normal data.

\subsection{Choice of normal data and computation of the Poincar\'e map}
\label{S:KAM:case2}

In order to evaluate the Poincar\'e map~\eqref{eq:Poin:map}, we need to compute the solution
\begin{align*}
r(s) = {} & r^{(0)}(s) + \ep r^{(1)}(s) + \ep^2 r^{(2)}(s) + \cO(\ep^\mu)\,,\\
\te(s) = {} & \te^{(0)}(s) + \ep \te^{(1)}(s) + \ep^2 \te^{(2)}(s) + \cO(\ep^\mu)\,,
\end{align*}
with initial conditions $r(0)=r_0$ and $\te(0)=\te_0$. The $0^\mathrm{th}$-order is
\[
r^{(0)}(s) = r_0\,, \qquad
\te^{(0)}(s) = \te_0 + T(s)\,, \qquad
T(s) := -\int_0^s \tau(\al) d\al\,.
\]

To this end, we introduce the notation
\begin{align}
&T_0:=T(\ell)=-\int_0^\ell \tau(\al) d\al\,,
\\&g(s) := T(s)-T_0 \frac{s}{\ell} = -\int_0^s \tau(\al) d\al + \frac{s}{\ell} \int_0^\ell \tau(\al) d\al\,,
\end{align}
and we observe that $g$ is an $\ell$-periodic function satisfying $g(0)=g(\ell)=0$.

In what follows we shall assume that the total torsion of the curve is rational, i.e. it satisfies the assumption
\begin{equation}\label{eq:T0pq}
T_0 = \frac{2 \pi p}{q}\,,
\end{equation}
for some coprime integers $p,q$. Notice that $T_0$ corresponds to the (degenerate) frequency of the unperturbed problem ($\ep=0$), so~\eqref{eq:T0pq} is a resonance condition. This assumption will be crucial to gain control
of the trajectories of the Beltrami field in order to create hyperbolic periodic orbits.

The normal datum $f$ chosen in Subsection~\ref{S.normal} will be constructed using functions $a_n$ of the form
\begin{equation}\label{eq:my:an}
a_n = \Gamma_n' - i n \tau \Gamma_n\,,
\end{equation}
$n\in\cN\subset \NN\cap [2,\infty]$, where $\Gamma_n\equiv\Gamma_n(\al)$ is a real $\ell$-periodic function, and $\Gamma_n'$ denotes the $\alpha$-derivative of $\Gamma_n$.

In order to obtain an expression for the Poincar\'e map that allows us to apply a subharmonic Melnikov method to create hyperbolic periodic orbits, we further assume that the functions $\Gamma_n$ satisfy the integral relations

\begin{align}
& \int_0^{q \ell} \kappa(\si) \tau(\si) \Gamma_n(\si) R_{n+k}(\si) d\si = 0\,, \quad
k \in \{-1,1\}\,, \quad
n \in \cN\,,
\label{SCond1:1} \\
& \int_0^{q \ell} \kappa'(\si) \Gamma_n(\si) R_{n+k}(\si) d\si = 0\,, \quad
k \in \{-1,1\}\,, \quad
n \in \cN\,,
\label{SCond1:2} \\
& \int_0^{q \ell} \Gamma_j(\si) \Gamma_n'(\si) R_{n+k}(\si) d\si = 0\,, \quad
k \in \{-j,j\}\,, \quad
(j,n) \in \cN\times\cN\,,
\label{SCond2} \\
& \int_0^{q \ell} \tau \Gamma_j(\si) \Gamma_n(\si) R_{n+k}(\si) d\si = 0\,, \quad
k \in \{-j,j\}\,, \quad
(j,n) \in \cN\times\cN\,, \quad
n \neq j\,.
\label{SCond3}
\end{align}
where the function $R_{n+k}$ is defined as
\[
R_{n+k}(\si) := e^{i (n+k) T(\si)} = e^{i (n+k)g(\si)} e^{i \frac{2 \pi p (n+k)
\si}{q \ell}} \,.
\]

Before stating the main result of this subsection, we prove the following instrumental lemma, which provides additional integral identities for the functions $\Ga_n$.

\begin{lemma}
Under the hypotheses
\eqref{SCond1:1} and
\eqref{SCond1:2}, we have
\begin{equation}\label{eq:Lema:int1}
\int_0^{q\ell} \kappa a_n R_{n+k} = 0\,, \quad
k \in \{-1,1\}\,, \quad
n \in \cN\,,
\end{equation}
and
\begin{equation}\label{eq:Lema:int2}
\int_0^{q\ell} \Gamma_j a_n R_{n+k} = 0\,, \quad
k \in \{-j,j\}\,, \quad
(j,n) \in \cN\times\cN\,, \quad
n \neq j\,.
\end{equation}
\end{lemma}

\begin{proof}
Using \eqref{eq:my:an} we have
\[
\int_0^{q\ell} \kappa a_n R_{n+k} =
\int_0^{q\ell} \kappa \Gamma_n' R_{n+k} -
\int_0^{q\ell} i n \kappa \tau \Gamma_n R_{n+k}
\]
and, using~\eqref{SCond1:1}, we observe that the second integral vanishes.
The fact that the first one also vanishes follows using integration by parts:
\[
\int_0^{q\ell} \kappa \Gamma_n' R_{n+k} =
-
\int_0^{q\ell} \Gamma_n (\kappa' R_{n+k} - i (n+k) \kappa \tau R_{n+k}) = 0.
\]
Then, Property~\eqref{eq:Lema:int1} is obtained using~\eqref{SCond1:1}
and~\eqref{SCond1:2}. Property~\eqref{eq:Lema:int2} is analogous.
\end{proof}

With all these hypotheses and using the
definition~\eqref{eq:Poin:map}, we can now compute a closed form for the $q$-th iterate of the Poincar\'e map $\Pi$:

\begin{proposition}\label{eq:Prop:Poin:res}
Assume that the assumptions~\eqref{eq:T0pq} and~\eqref{SCond1:1}--\eqref{SCond3} hold. Then, we have
\[
\Pi^{q}(r_0,\theta_0)=
\begin{pmatrix}
r_0 + \ep^{\mu} \Pi_r^{(\mu)}(r_0,\te_0) + \cO(\ep^3) \\
\te_0 + \omega(r_0) + \ep^{\mu}  \Pi_\te^{(\mu)}(r_0,\te_0) + \cO(\ep^3)
\end{pmatrix}\,,
\]
where
\begin{align*}
\omega(r_0) = {} &  2\pi p - q\ep^2 \left( \frac{12-r_0^2}{32} \int_0^{\ell} \tau \kappa^2 - \sum_{n \in \cN} 4
n(n-1) r_0^{2n-4} \int_0^{\ell} \tau \Gamma_n^2 \right)\,, \\
\Pi_r^{(\mu)}(r_0,\te_0) = {} & \int_0^{q\ell} v_r^{(\mu)} (s,r_0,\te_0+T(s)) ds \,, \\
\Pi_\te^{(\mu)}(r_0,\te_0) = {} & \int_0^{q\ell}
(v_\te^{(\mu)} (s,r_0,\te_0+T(s))
+ \tau(s) v_\al^{(\mu)} (s,r_0,\te_0+T(s))) ds \,.
\end{align*}
\end{proposition}

\begin{remark}
In Section~\ref{S.Proof} we will apply a subharmonic Melnikov method to show that
the terms of order $\ep^{\mu}$ can be selected to destroy some resonant invariant tori,
thus creating hyperbolic (and elliptic) periodic orbits in the interior of the tube.
\end{remark}

\begin{proof}[Proof of Proposition~\ref{eq:Prop:Poin:res}]
For the sake of clarity, it is convenient to use complex exponents. Using complex
exponentials and expanding the terms
$u_r^{(1)}$,
$u_\te^{(1)}$,
$u_r^{(2)}$ and
$u_\te^{(2)}$, the scaled vector field $X$ (see Proposition~\ref{prop:BeltramiX}) reads
\begingroup
\allowdisplaybreaks
\begin{align*}
X_r^{(1)} = {} & \frac{3(r^2-1)}{16} \left( (\tau \ka i - \ka') e^{i \te} + (-\tau \ka i -\ka') e^{-i\te} \right) \\
          &  + \sum_{n \in \cN} r^{n-1} (a_n e^{i n \te}+\overline{a_n} e^{-i n \te})\,,\\
X_\te^{(1)} = {} & \frac{r^2-3}{16 r} \left(  (-\tau \ka - \ka' i) e^{i \te} + (-\tau \ka + \ka' i) e^{-i\te}\right) \\
            & + \sum_{n \in \cN} i r^{n-2} (a_n  e^{i n \te}-\overline{a_n} e^{-i n \te})\,, \\
X_r^{(2)} = {} &
\frac{3(r^3-r)}{8} \ka \ka'
- \frac{3(r^3-r)}{16}
 \left( (\tau \ka^2 i - \ka \ka') e^{i 2 \te}  + (-\tau \ka^2 i -\ka \ka') e^{-i 2 \te}
 \right) \\
&
- \sum_{n \in \cN} r^{n}(
\ka a_n  e^{i (n+1) \te}+ \ka \overline{a_n} e^{-i (n+1) \te}
) \\
& - \sum_{n \in \cN} r^{n}(
\ka a_n  e^{i (n-1) \te}+ \ka \overline{a_n} e^{-i (n-1) \te}
) \\
& + \frac{13(r^3-r)}{48}\left( (\tau \ka^2 i - \ka \ka') e^{i 2\te} + (-\tau \ka^2 i - \ka \ka') e^{-i 2\te} \right) + H_1 \\
& + \sum_{n \in \cN} \frac{ (n+1) r^n - (n+1) r^{n-2}}{4 n}
              \left( \ka a_n e^{i (n-1) \te} + \ka \overline{a_n} e^{-i (n-1) \te} \right) \,, \\
X_\te^{(2)} = {} &
 \frac{r^2-3}{8} \tau \ka^2
-\frac{r^2-3}{16}
\left(
(-\tau \ka^2 - \ka \ka' i) e^{i 2 \te} + (-\tau \ka^2 + \ka \ka'i) e^{- i 2 \te}
\right) \\
& - \sum_{n\in \cN} r^{n-1}
\left(
i \ka a_n e^{i (n+1)\te} - i \ka \overline{a_n} e^{-i (n+1)\te}
\right) \\
& - \sum_{n\in \cN} r^{n-1}
\left(
i \ka a_n e^{i (n-1)\te} - i \ka \overline{a_n} e^{-i (n-1)\te}
\right) \\
& + \frac{13(r^2-2)}{96}
\left(
(-\tau \ka^2 - \ka \ka'i) e^{i 2 \te}+(-\tau \ka^2 + \ka \ka'i) e^{-i 2 \te}
\right) \\
& + \sum_{n \in \cN}
\frac{(n-1) r^{n-1}-(n+1) r^{n-3}}{4 n}
\left(
i \ka a_n e^{i (n-1) \te} - i \ka \overline{a_n} e^{-i (n-1) \te}
\right)\,.
\end{align*}
\endgroup
Now we integrate the trajectories of the above vector field. The procedure is
analogous to the computations performed in the Appendix. The first order
terms are readily obtained:
\begingroup
\allowdisplaybreaks
\begin{align*}
r^{(1)}(s) = {} & \int_0^s X_r^{(1)}[\si] d\si \\
= {} & \frac{3(r_0^2-1)}{16}
\left(
- \ka e^{i (\te_0+T)} - \ka e^{-i (\te_0+T)}
+ \ka(0) e^{i \te_0} + \ka(0) e^{-i \te_0}
\right) \\
& + \sum_{n \in \cN} r_0^{n-1}
\left(
\Gamma_n e^{i n(\te_0+T)} + \Gamma_n e^{-in (\te_0+T)}
+ \Gamma_{n}(0) e^{i n\te_0} + \Gamma_{n}(0) e^{-i n\te_0}
\right) \,,\\
\te^{(1)}(s) = {} & \int_0^s X_\te^{(1)}[\si] d\si \\
= {} & \frac{r_0^2-3}{16 r_0}
\left(
- \ka i e^{i (\te_0+T)} + \ka i e^{-i(\te_0+T)}
+ \ka(0) i e^{i \te_0} - \ka(0) i e^{-i \te_0}
\right) \\
& + \sum_{n \in \cN} r_0^{n-2}
\left(
i \Gamma_n e^{i n(\te_0+T)}-i \Gamma_n e^{-in(\te_0 +T)}
-i \Gamma_{n}(0) e^{i n\te_0}+i \Gamma_{n}(0) e^{-i n\te_0}
\right) \,.
\end{align*}
\endgroup
Using the resonance condition~\eqref{eq:T0pq} and the periodicity of $\kappa$ and
$\Gamma_n$, we conclude that
\[
r^{(1)}(q \ell) = \te^{(1)}(q \ell) = 0 \,.
\]

The terms $r^{(2)}(s)$ and $\te^{(2)}(s)$ are obtained by reproducing the
computations in the Appendix (see Equation~\eqref{eq:comp:theta2}), but with a more involved integrand. For
example, we have
\begingroup
\allowdisplaybreaks
\begin{align*}
& \frac{\pd X_\te^{(1)}}{\pd r}[s] r^{(1)}(s)
=
 \frac{3(r^2+3)(r^2-1)}{128 r^2}
\tau \ka^2 \\
\quad & +
\frac{3(r^2+3)(r^2-1)}{256 r^2}
\ka(0) (e^{i\te_0}+e^{-i\te_0})
\left(
(-\tau \ka-\ka' i)
e^{i(\te_0+T)}
+
(-\tau \ka+\ka' i)
e^{-i(\te_0+T)}
\right) \\
\quad & +
\frac{3(r^2+3)(r^2-1)}{256 r^2}
\left(
(\tau \ka^2 + \ka \ka' i) e^{i 2(\te_0 + T)}
+
(\tau \ka^2 - \ka \ka' i) e^{-i 2(\te_0 + T)}
\right) \\
\quad & +
\sum_{n \in \cN} \frac{(r^2+3) r^{n-3}}{16}
\left(
(-\tau \ka - \ka' i) \Gamma_n e^{i (n+1) (\te_0+T)}
+
(-\tau \ka + \ka' i) \Gamma_n e^{-i (n+1) (\te_0+T)}
\right) \\
\quad & +
\sum_{n \in \cN} \frac{(r^2+3) r^{n-3}}{16}
\left(
(-\tau \ka + \ka' i) \Gamma_n e^{i (n-1) (\te_0+T)}
+
(-\tau \ka - \ka' i) \Gamma_n e^{-i (n-1) (\te_0+T)}
\right) \\
\quad & +
\sum_{n \in \cN} \frac{(r^2+3) r^{n-3}}{16}
\Gamma_{n}(0) (e^{i n \te_0}+e^{-i n \te_0})
\left(
(-\tau \ka - \ka' i) e^{i (\te_0+T)}
+
(-\tau \ka + \ka' i) e^{-i (\te_0+T)}
\right) \\
\quad & +
\sum_{n \in \cN} \frac{i 3(n-2)(r^2-1)r^{n-3}}{16}
\left(
-\ka a_n e^{i(n+1) (\te_0+T)}
+\ka \overline{a_n} e^{-i(n+1)(\te_0+T)}
\right) \\
\quad & +
\sum_{n \in \cN} \frac{i 3(n-2)(r^2-1)r^{n-3}}{16}
\left(
-\ka a_n e^{i(n-1) (\te_0+T)}
+\ka \overline{a_n} e^{-i(n-1)(\te_0+T)}
\right) \\
\quad & +
\sum_{n \in \cN} \frac{i 3(n-2)(r^2-1)r^{n-3}}{16}
\ka(0) (e^{i\te_0}+e^{-i \te_0})
\left(
a_n e^{i n (\te_0+T)}
-\overline{a_n} e^{-i n(\te_0+T)}
\right)\\
\quad & +
\sum_{(n,j) \in \cN^2} i(n-2) r^{n+j-4}
\left(
\Gamma_j a_n e^{i(n+j)(\te_0+T)}
-
\Gamma_j \overline{a_n} e^{-i(n+j)(\te_0+T)}
\right)\\
\quad & +
\sum_{(n,j) \in \cN^2} i(n-2) r^{n+j-4}
\left(
\Gamma_j a_n e^{i(n-j)(\te_0+T)}
-
\Gamma_j \overline{a_n} e^{-i(n-j)(\te_0+T)}
\right)\\
\quad & +
\sum_{(n,j) \in \cN^2} i(n-2) r^{n+j-4}
\Gamma_{j}(0)(e^{i j \te_0}+e^{-i j \te_0}) \left(
a_n e^{i n (\te_0+T)}
-
\overline{a_n} e^{-i n (\te_0+T)}
\right)\,.
\end{align*}
\endgroup
Most of the above terms will not contribute to the final computation.
As an illustration, we notice that
\begin{align*}
& \int_0^s (-\tau \ka^2 - \ka \ka'i) e^{i 2 (\te_0+T)} d \si =
\frac{1}{2} \int_0^s (- \ka^2 i e^{i 2 (\te_0+T)})' d \si \\
& \qquad = \frac{1}{2}
\left(
- \ka(s)^2 i e^{i 2 (\te_0+T(s))} + \ka(0)^2 i e^{i 2 \te_0}
\right)\,,
\end{align*}
which vanishes at $s=q \ell$,
due to the resonance condition~\eqref{eq:T0pq}.
Hence, the only term that contributes in the sum of indexes $(n,j) \in \cN\times\cN$ is the
term $n=j$, and we observe that
\[
\int_0^{q\ell}
\left(\Gamma_n(\si) a_n(\si)-\Gamma_n(\si) \overline{a_n}(\si) \right) d\si
= -2 n i \int_0^{q\ell} \tau(\si) \Gamma_n(\si)^2 d\si\,.
\]
Putting together all the terms, we obtain
\[
\te^{(2)}(q\ell) =
-\frac{12-r_0^2}{32} \int_0^{q\ell} \tau \kappa^2 + \sum_{n \in \NN} 4
n(n-1) r_0^{2n-4} \int_0^{q\ell} \tau \Gamma_n^2\,,
\]
and $r^{(2)}(q\ell) = 0$. Since the functions that are integrated in this expression are $\ell$-periodic, the statement follows.
\end{proof}

\subsection{Action-angle variables and KAM}
\label{S:KAM:case1}

It is straightforward to check that the iterated Poincar\'e map $\Pi^q$ can be written using the action-angle variables $I:=r_0^2/2$ and $\phi:=\theta_0$ as
\begin{equation}\label{PoinAA}
\Pi^{q}(I,\phi)=
\begin{pmatrix}
I + \ep^\mu \Pi_I^{(\mu)} (I,\phi) + \cO(\ep^3) \\
\phi+\hat\omega(I)
+ \ep^\mu \Pi_\phi^{(\mu)} (I,\phi)
+ \cO(\ep^3)
\end{pmatrix}\,,
\end{equation}
where the frequency function $\hat\omega(I)$ is
\begin{equation}\label{eq:trunc:freq}
\hat\omega(I) := 2\pi p - q\ep^2 \left( \frac{6-I}{16} \int_0^{\ell} \tau \kappa^2 - \sum_{n \in \cN} 2^n
n(n-1) I^{n-2} \int_0^{\ell} \tau \Gamma_n^2 \right)\,,
\end{equation}
and the terms of order $\cO(\ep^\mu)$, $2<\mu<3$, are given by
\begin{align*}
\Pi_I^{(\mu)}(I,\phi) = {} & \sqrt{2I}\int_0^{q\ell} v_r^{(\mu)}
(s,
{\sqrt{2I}},
\phi+T(s)) ds \,, \\
\Pi_\phi^{(\mu)}(I,\phi) = {} & \int_0^{q\ell}
(v_\te^{(\mu)} (s,\sqrt{2I},\phi+T(s))
+ \tau(s) v_\al^{(\mu)} (s,\sqrt{2I},\phi+T(s))) ds \,.
\end{align*}
The action $I$ takes values in the interval $(0,1/2)$.

The existence of invariant tori of the Beltrami field $v$ in the tube $\cT_\ep$ follows from a standard
application of Moser's theorem to the iterated Poincar\'e map $\Pi^q$. Specifically, using the measure $d\al\, dy$ on
$\SS^1_\ell\times\DD$ so that $|\SS^1_\ell\times\DD|=\pi\ell$, one can state the result as follows (compare with the analogous KAM theorem for harmonic fields in the Appendix, c.f. Proposition~\ref{P.tori}). Notice that the statement is independent of the particular form of the $\cO(\ep^\mu)$ perturbation.

\begin{theorem}\label{T:KAM}
Take any $\delta'>0$. Suppose that the torsion satisfies the resonant condition~\eqref{eq:T0pq}, and the twist condition
\[
\cA:=\int_0^\ell \tau(\al)\Big(\frac{\kappa(\al)^2}{16}
+48\Gamma_3{(\al)}^2\Big)\, d\al\neq0\,.
\]
Then for small enough~$\ep$ the Beltrami field~$v$ constructed in this section has a set of
invariant tori of the form
$$
\{I + \cO(\ep^\mu)=\text{constant} \}\,,
$$
contained in $\SS^1_\ell\times\DD$,
whose measure is at least $\pi\ell-\delta' -C\ep^{\frac \mu2-1}$.
\end{theorem}

\begin{proof}
By Proposition~\ref{P:area}, the Poincar\'e map~\eqref{eq:Poin:map} $\Pi:\DD_R\to \DD$, $R<1$, at the section ${\alpha=0}$, of the Beltrami field $v$, preserves an area measure. The same holds for the iterated map $\Pi^q$, which in action-angle coordinates $(I,\phi)$ reads as an $\cO(\ep^\mu)$ perturbation of an integrable twist map, c.f. Equation~\eqref{PoinAA}, with frequency function $\hat\omega(I)$.

Moser's twist condition reads as
\[
\hat\om'(I)=q\ep^2(\cA+\cO(I))\neq 0\,.
\]
If the quantity $\cA$ is nonzero, then being $\hat\om'(I)$ a polynomial, it vanishes in a finite set $\{I_1,\dots,I_L\}$ of points of the interval $(0,1/2)$. Take a real $\delta''>0$. For each $I$ in the complement
$$(0,1/2)\backslash \bigcup_{k=1}^L(I_k-\delta'',I_k+\delta'')\,,$$
the twist is then a nonzero constant of order $\ep^2$ bounded from below by a constant that depends on $\de''$, and since the iterated Poincar\'e map is an order $\cO(\ep^\mu)$
perturbation of an integrable twist map, it follows (see
e.g.~\cite{Siegel}) that for small enough $\ep$ all the disk $\DD_R$ but a
set of measure at most $C\de''+C\ep^{\frac\mu2-1}$ is covered by quasi-periodic invariant
curves of the iterated Poincar\'e map, which in turn yields quasi-periodic invariant curves for the Poincar\'e map $\Pi$. The constant $C$ depends on $\de''$ but not on $\ep$. In terms of the Beltrami field~$v$, this obviously means that the whole tube $\es\times\DD$ but a set of measure at most
$\de'+C\ep^{\frac\mu2-1}$ is covered by ergodic invariant tori of~$v$, as claimed, where we are setting $\delta':=C\delta''$.
\end{proof}

\begin{remark}
  With the same proof, but with a more thorough bookkeeping,
  one can in fact 
{assume}
a slightly more general twist
  condition. Namely, the theorem remains true under the weaker
  assumptions that the coefficients $(\cA_n)_{n\geq
3}$ are not all zero, where
\begin{align*}
  \cA_n:=
  \begin{cases}
     \int_0^\ell \tau(\al)\Big(\frac{\kappa(\al)^2}{16} +48\Gamma_3(\al)^2\Big)\,
d\al & \text{if } n=3\,, \\[2mm]
 \int_0^\ell \tau(\al) \Gamma_n(\al)^2 d\al & \text{if } n\neq3\,.
\end{cases}
\end{align*}
\end{remark}

\section{Proof of the main theorem}
\label{S.Proof}

In this section we prove the following theorem. The fact that this result
implies Theorem~\ref{T.main} from the Introduction follows from two observations. First, any collection of tubes $\cT_1,\dots,\cT_N$ is isotopic to a collection of thin tubes $\cT_\ep(\gamma_1),\dots,\cT_\ep(\gamma_N)$ for any small enough $\ep$, where $\gamma_j$ is a representative curve of the core knot of $\cT_j$. Second, if $u$ is a Beltrami field in $\RR^3$ satisfying $\curl u=\lambda u$, then the rescaled field $u'(x):=u(\la^{-1}\la'x)$ is a Beltrami field satisfying $\curl u'=\la'u'$.

\begin{theorem}\label{T.Acta2}
Let $\gamma_1, \ldots, \gamma_N$ be pairwise disjoint (possibly knotted and
linked) closed curves in $\RR^3$.
For small enough $\ep$, for any $\de'>0$ and for any nonzero constant $\la=\cO(\ep^3)$, one can transform the
collection of pairwise disjoint thin tubes $\cT_\ep(\gamma_1), \ldots,
\cT_\ep(\gamma_N)$ by a diffeomorphism $\Phi$ of $\RR^3$, arbitrarily close to
the identity in the $C^{2,\beta}$ norm ($\beta<1$), so that $\Phi[\cT_\ep(\gamma_1)], \ldots,
\Phi[\cT_\ep(\gamma_N)]$ are invariant solid tori of a Beltrami field
satisfying the equation $\curl u = \lambda u$ in $\RR^3$ and the decay
condition $|u(x)|<\frac C{1+|x|}$. Moreover, the interior of each $\Phi[\cT_\ep(\gamma_j)]$ contains a set of
ergodic invariant tori~$\cI_j$, of measure greater than $(1-\de'-C\ep^{1/4}) |\Phi[\cT_\ep(\ga_j)]|$, and at least $M$ hyperbolic periodic orbits that are isotopic to each other and cablings of the curve~$\gamma_j$.
\end{theorem}

\begin{proof}
  The proof is divided in four steps:

  \subsubsection*{Step 1: Existence of the curve and the normal data}

For the ease of notation, we will
drop the subscript~$j$ until we consider all the tubes simultaneously
in the fourth step of the proof. Consider each smooth curve $\ga\equiv
\ga_j$ and the tube $\cT_\ep\equiv \cT_\ep(\ga)$ as
defined in Subsection~\ref{S.Belt1}.
We shall assume that~$\ga$ satisfies the condition
\begin{equation}\label{hypothesis1}
\int_0^\ell \kappa(\al)^2 \tau(\al) \, d\al\neq0\,,
\end{equation}
and the resonant condition~\eqref{eq:T0pq}, which hold for a dense subset of smooth curves in
the $C^k$~topology, with $k\geq3$ (see~\cite[Lemma 7.9]{Acta}).

Now we want to show that there exist $\ell$-periodic functions $\Ga_n$, $n\in\cN$, introduced in Subsection~\ref{S:KAM:case2}, which satisfy the integral relations~\eqref{SCond1:1}--\eqref{SCond3}. To this end, we consider the (finite-dimensional) vector space $\mathcal S$ spanned by linear combinations (with real coefficients) of the finite set of functions
\begin{equation}\label{Eq.finite}
\bigcup_{n\in\cN}\{\ka\tau R_{n+1},\ka\tau R_{n-1}, \ka' R_{n+1},\ka' R_{n-1}\}\,.
\end{equation}
(The functions $R_j$ were defined in Subsection~\ref{S:KAM:case2}). Then we can take smooth functions $\Ga_n$, $n\in\cN$, with disjoint supports, and in the (infinite-dimensional) $L^2$-orthogonal complement of $\mathcal S$. An easy computation shows that these functions satisfy all the conditions~\eqref{SCond1:1}--\eqref{SCond3} except for the condition in~\eqref{SCond2} given by
\[
\int_0^{q\ell}\Ga_n\Ga_n'R_{2n}=0\,,
\]
which is readily seen to be equivalent (integrating by parts) to
\begin{equation}\label{Eqtau}
\int_0^{q\ell}\tau\Ga_n^2R_{2n}=0\,.
\end{equation}

To ensure the existence of $\Ga_n$ satisfying this assumption, we construct a curve $\ga_1$ which is a small perturbation of $\gamma$, whose torsion changes sign many times on some a priori prescribed intervals of $\SS^1_\ell$ (which in turn will be the supports of the functions $\Ga_n$). Indeed,
let $f_n(\al)$ be a smooth function that is supported in an
interval $K_n$ of the circle
$\SS^1_\ell$, $n\in\cN$, and take a small but fixed constant~$\eta$. The intervals $K_n$ are taken pairwise disjoint. We now transform the
curve~$\ga$ by adding a rapidly-rotating small-amplitude helicoid
supported in the above portion of the circle chosen so that the
perturbation is $C^{2,\be}$-small for all $\be<1$ but not $C^3$-small:
\[
\ga_1(\al):=\ga(\al)- \eta^3\, f_n(\al)\, \bigg( e_1(\al)\,\cos
\frac\al\eta + e_2(\al)\,\sin \frac\al\eta\bigg)\,.
\]
We recall that $\al$ is an arc-length parametrization of~$\ga$ and $e_1(\al)$, $e_2(\al)$ are the normal and binormal
vector fields (see Subsection~\ref{S.Belt1}) on the curve~$\ga$. Notice that~$\al$ is almost an
arc-length parametrization of~$\ga_1$ in the sense that
\[
\bigg|\frac{d\ga_1}{d\al}\bigg|^2=1+\cO_0(\eta^2)\,.
\]
This enables us to carry out
computations in the parameter~$\al$ as if it were an arc-length
parameter of~$\ga_1$, up to errors of order~$\eta^2$. We will not mention this explicitly in what follows for the
ease of notation. In this section, we shall denote by $\cO_0(\eta^j)$ any quantity $q(\al,y)$
defined on $\es\times\DD$ such that
\[
\|q\|_{C^0(\es\times\DD)}\leq C\eta^j\,,
\]
where the constant $C$ does not depend on~$\eta$. Note that we are not
making any assumptions about the size of the derivatives of~$q$.

It is easy to check that
\begin{align*}
\ga_1'(\al)&= \ga'(\al) + \cO_0(\eta^2)\,,\\
\ga_1''(\al)&= \ga''(\al)+\eta f_n(\al)\, \bigg( e_1(\al)\,\cos
\frac\al\eta + e_2(\al)\,\sin \frac\al\eta\bigg)+\cO_0(\eta^2)\,,\\
\ga_1'''(\al)&= \ga'''(\al)-f_n(\al)\, \bigg( e_1(\al)\,\sin
\frac\al\eta - e_2(\al)\,\cos \frac\al\eta\bigg)+\cO_0(\eta)\,.
\end{align*}
Using the formulas for the Frenet frame, this shows that the curvature
and torsion of the curve~$\ga_1$ are
\begin{align*}
\ka_1(\al)&=\frac{|\ga_1'\times\ga_1''|}{|\ga_1'|^3}=
            \ka(\al)+\cO(\eta)\,,\\
 \tau_1(\al)&=\frac{\ga_1'\times\ga_1''\cdot\ga_1'''}{|\ga_1'\times\ga_1''|^2}=\tau(\alpha)+\frac{f_n(\al)}{\ka(\al)}\cos\frac\al\eta
 +\cO(\eta)\,.
\end{align*}
Notice that $\ka_1>0$
everywhere for small enough~$\eta$.

It is then clear that one can choose an $\eta$-independent $f_n$  so
that $\tau_1$ changes sign $\lfloor c\eta^{-1}\rfloor$ times in the
interval $K_n$, where~$c$ is a uniform constant. (For example, one can take $f_n:=2\tau\ka \chi_n$,
where $\chi_n$ is a function supported on $K_n$ which is equal to~1 in an open interval of $K_n$.) One can then take the curve
$$
\ga(\al)- \eta^3\, \sum_{n\in\cN} f_n(\al)\, \bigg( e_1(\al)\,\cos
\frac\al\eta + e_2(\al)\,\sin \frac\al\eta\bigg)\,,
$$
which is a $C^{2,\beta}$-small perturbation of $\ga$ in each $K_n$. The condition~\eqref{hypothesis1} is still satisfied because
\begin{align*}
 \int_0^\ell \ka_1(\al)^2\,\tau_1(\al)\,
  \bigg|\frac{d\ga_1}{d\al}\bigg|\, d\al
&=\int_0^\ell
  \bigg( \ka(\al)^2\tau(\alpha)+\ka(\alpha) \sum_{n\in\cN}f_n(\al)\,\cos\frac\al\eta\bigg)\,
  d\al+\cO(\eta)\\
&=\int_0^\ell \ka(\al)^2\tau(\alpha)\, d\al+\cO(\eta)\,,
\end{align*}
which is nonzero for all small
enough~$\eta$. Here we have used that the integral of any smooth
function~$F$ satisfies $|\int_0^\ell F(\al)\,
\cos\frac\al\eta\, d\al|< C_N\eta^N$ for any~$N$.

For simplicity, in what follows we shall still call $\ga$ this new curve. We can also safely assume that the resonant condition~\eqref{eq:T0pq} holds (for $p$ and $q$ different from those of the original curve $\gamma$), because this condition is satisfied for a dense subset of smooth curves in
the $C^k$~topology, with $k\geq3$ (notice that a $C^3$-small deformation of the curve does not affect the previous properties).

Therefore, we can take the functions $\Gamma_n$, $n\in\cN$, supported
in the intervals $K_n$, and in the $L^2$-orthogonal complement of the
finite vector space $\cS$ defined in~\eqref{Eq.finite}. Since the
torsion $\tau$ changes sign $\lfloor c\eta^{-1}\rfloor$ times in each
support of $\Gamma_n$, it is easy to check that $\Gamma_n$ can be
chosen so that the condition~\eqref{Eqtau} is satisfied. This
completes the proof of the existence of the functions $\Gamma_n$
satisfying all the integral
relations~\eqref{SCond1:1}--\eqref{SCond3}.

\subsubsection*{Step 2: The resonant tori of the local Beltrami field}

Our next goal is to show that, given
{$M$}
points $0<I_1<I_2<\cdots <
I_M<\frac12$, one can choose the set~$\cN$ and the functions~$\Ga_n$
($n\in\cN$) such that the frequency function
\[
\hat\omega(I) = 2\pi p - q\ep^2 \left( \frac{6-I}{16} \int_0^{\ell} \tau \kappa^2 - \sum_{n \in \cN} 2^n
n(n-1) I^{n-2} \int_0^{\ell} \tau \Gamma_n^2 \right)\,,
\]
defined in~\eqref{eq:trunc:freq}, {satisfies
\[
\hat \omega(I_k)= 2\pi p
\qquad
\mbox{and}
\qquad
\hat \omega'(I_k) \neq 0
\]
for all $1\leq k\leq M$}.
By construction, this means that the curves
$\{I=I_k\}$ are approximate resonant curves of the Poincar\'e map~$\Pi$ consisting
of points that are $q$-periodic modulo an error of
size~$\cO(\ep^\mu)$. This gives approximate resonant invariant tori of the Beltrami field $v$.

Indeed, it is not hard to see that, as $\tau$ and $\tau R_{2n}$ are
not proportional, one can choose the functions~$\Ga_n$ as in Step~1 (which
satisfied $\int_0^\ell \tau \Ga_n^2 R_{2n}=0$) so that
\[
\int_0^\ell \tau \Ga_n^2\neq0\,.
\]
If we now take a set~$\cN$ of large enough cardinality and notice that
all the above conditions still hold if
one {replaces} $\Ga_n$ by $c_n\Ga_n$, where $c_n$ is any real
constant, it becomes apparent that one can indeed prescribe as many nondegenerate
zeros of the polynomial $\hat\om(I)-2\pi p$ as one wishes.

\subsubsection*{Step 3: A subharmonic Melnikov method to create hyperbolic periodic orbits}

The goal of this step is to select a suitable perturbation of size $\cO(\ep^\mu)$ in
order to create hyperbolic periodic orbits in the interior of $\cT_\epsilon$. To this end, we will
resort to a subharmonic Melnikov criterion adapted to the small twist and the
perturbation size.

Consider the approximate resonant tori $\{I=I_k\}$ constructed in Step~2. We can express
the frequency function $\hat \omega(I)$,  in a
neighborhood of~$\{I=I_k\}$, as
\begin{equation}\label{eq:taylor:om}
\begin{split}
\hat \om(I) = {} &  2 \pi p + \hat \om'(I_k)(I-I_k) + \cO(\ep^2 (I-I_k)^2))\,,
\end{split}
\end{equation}
and we will be concerned with values of~$I$ such that
\begin{equation}\label{eq:ansatzI}
I-I_k = \cO(\ep^\beta)\,,
\end{equation}
where $0<\beta<1$ is a coefficient that will be specified later.

The following result gives an explicit sufficient condition for the existence
of hyperbolic periodic orbits of the vector field~$v$.

\begin{theorem}\label{T:subh:Melnikov}
Let $\{I=I_k\}$ be the approximate resonant invariant torus of $v$ constructed above.
Suppose that the subharmonic Melnikov
function that we define as
\begin{equation}\label{eq:sub:Mel2}
M^{q/p}(\phi):=
\int_0^{q \ell} v^{(\mu)}_r(s,\sqrt{2I_k},\phi+T(s)) \, ds
\end{equation}
has $j$ zeros in $\phi \in [0,2\pi)$, all of which are
non-degenerate. Then
for small enough~$\ep$, the Beltrami field~$v$ has $j$ periodic orbits
contained in a small neighborhood of the tori $\{I=I_k\}$. These orbits
wind around the core knot of the torus $q$~times in the
coordinate~$\al$ and $p$~times in the coordinate~$\theta$. Moveover,
$j$ is even and
half of these periodic orbits are hyperbolic and the other half are elliptic.
\end{theorem}

\begin{proof}
The proof is independent of the choice of $I_k$. For simplicity, let
us set
\[
A:=\ep^{-2}\hat \om'(I_k)\neq 0\,.
\]
Our goal is to show that if
the function $M^{q/p}(\phi)$ has $j$ non-degenerate zeros
$\{\phi^1,\dots, \phi^j\}\subset [0,2\pi)$,
then there are exactly
$j$ periodic orbits $\{(\al_i(s),I_i(s),\phi_i(s))\}_{i=1}^j$ of period~$q\ell$
contained in a toroidal annulus of the form
$$
|I-I_k|<C(\ep^{\mu-2}+\ep^{3-\mu}) <C\ep^\be\,,
$$
where $\be:=\min\{\mu-2,3-\mu\}>0$. The best possible choice
is~$\mu:=5/2$, so the reader may assume without loss of generality
that $\be=1/2$.
Furthermore, these periodic orbits pass
close to the zeros of the Melnikov function in the sense that
\[
\min_{0\leq s\leq q\ell} |\phi_i(s)-\phi^i|< C\ep^\be\,.
\]

To prove this, notice that the expression~\eqref{PoinAA} for the
iterated Poincar\'e map $\Pi^q$ is
\[
\Pi^q(I,\phi)=\big(I+\ep^\mu F(I,\phi)+\cO(\ep^3),\,
\phi+\hat \om(I)+ \ep^\mu G(I,\phi)+\cO(\ep^3)\big)
\]
with
\begin{align*}
F(I,\phi)&:= \sqrt{2 I}\int_0^{q \ell}
v^{(\mu)}_r(s,\sqrt{2I},\phi+T(s)) \, ds\,,\\
G(I,\phi)&:= \int_0^{q \ell}
(v^{(\mu)}_\te(s,\sqrt{2I},\phi+T(s))
+ \tau(s) v^{(\mu)}_\al(s,\sqrt{2I},\phi+T(s)))\, ds\,.
\end{align*}

Characterizing
the periodic trajectories of period $q\ell$ is
equivalent to identifying the fixed points of the iterated Poincar\'e
map $\Pi^q$. Hence, denoting the $(I,\phi)$ components of $\Pi^q$ by
$(\Pi^q_I,\Pi^q_\phi)$, we will look for zeros of the function
\[
\cR(I,\phi):=\big({\Pi^q_I(I,\phi)-I}, {\Pi^q_\phi(I,\phi)-\phi}\big)\,.
\]
Of course, as the angle $\phi$ takes values in $\SS^1$, the second
component of the above map must be understood modulo $2\pi$. Using the
expressions~\eqref{eq:taylor:om}
and~\eqref{eq:ansatzI}, we write
\begin{align}\label{cR}
\cR(I,\phi):=\big( \ep^\mu F(I,\phi)+\cO(\ep^3),\, \ep^2 A\,
  (I-I_k)+\ep^\mu G(I,\phi)+\cO(\ep^3 + \ep^{2+2\beta})\big)\,.
\end{align}

Before discussing the zeros of this map, let us derive
estimates for the various functions that appear in its definition. An
obvious estimate is that
\[
G(I,\phi)=\cO(1)\,,
\]
which implies that the leading term in the second component of the
map is simply $\ep^2 A(I-I_k)$.
On the other hand, the subharmonic Melnikov function dominates the $I$ component:
\[
F(I,\phi) = F(I_k,\phi) + \cO(\ep^\beta) =\sqrt{2I_k} M^{q/p}(\phi) + \cO(\ep^\beta)\,.
\]

Let us suppose that $(\widetilde I_k,\widetilde \phi)$ is a solution to the equation
$\cR(\widetilde I_k,\widetilde \phi)=0$ with $\widetilde I_k$ close to~$I_k$.
Since the leading term of the
$I$-component of the function $\cR$ is $\ep^\mu \sqrt{2 I_k} M^{q/p}(\phi)$,
the angle $\tilde\phi$ must be close to one of the zeros
$\phi^1,\dots,\phi^j$ of $M^{q/p}$.

The zeros of $\cR$ obviously coincide with those of
\[
\widetilde\cR(I,\phi):=\bigg(\frac{\cR_I(I,\phi)}{\ep^\mu},\frac{\cR_\phi(I,\phi)}{\ep^2}\bigg)\,,
\]
so let us now apply the implicit function theorem to
$\widetilde\cR(I,\phi)$ around the point $(I_k,\phi^i)$. Observe that
\[
\widetilde\cR(I_k,\phi^i)=\cO(\ep^\beta)
\]
and
\[
\det D\widetilde\cR(I_k,\phi^i)
= -A \sqrt{2 I_k} \, (M^{q/p})'(\phi^i)+\cO(\ep^\beta)\,,
\]
which is nonzero for small $\ep$ because $\phi^i$ is a non-degenerate
zero of the Melnikov function. The implicit function theorem then
guarantees the existence of a unique zero of $\widetilde\cR$ in a
neighborhood of $(I_k,\phi^i)$ of radius
$\cO(\ep^{\beta})$. This shows that there are exactly
$j$~periodic orbits of period $q\ell$ in a small neighborhood of the
approximate resonant torus $\{I=I_k\}$.

Since the function $M^{q/p}(\phi)$ is periodic and
$(M^{q/p})'(\phi^i)$ is nonzero at the points where this function
vanishes, it follows from the Hopf index theorem that there must be the same number of zeros
where this derivative is positive as zeros where it is negative. Indeed, the
index of the zero $\phi^i$ is given by the sign of
$(M^{q/p})'(\phi^i)$ and the sum of the indices must be equal to the
Euler characteristic of the circle, which is zero.

To identify which kind of orbits we have, let us compute the
eigenvalues $\La_\pm$ of the derivative of $\Pi^q$ at the zeros
$(\widetilde I_k,\widetilde \phi^i)$
of~$\cR$. The equation for the eigenvalues is
\begingroup
\renewcommand*{\arraystretch}{1.5}
\begin{align*}
0&=\det (D\Pi^q(\widetilde I_k,\widetilde\phi^i)-\La \mathbf 1)\\
&= \det \left(\begin{matrix}
(1-\La)+\ep^\mu  \frac{\pd F}{\pd  I}+\cO(\ep^3)  &  \ep^\mu \frac{\pd F}{\pd\phi} +\cO(\ep^3)\\
q\, \hat\om'(\widetilde I_ k)+\ep^\mu \frac{\pd G}{\pd I}+\cO(\ep^3) & (1-\La) +\ep^\mu
\,\frac{\pd G}{\pd\phi} +\cO(\ep^3)
\end{matrix}\right)\\
&= \det \left(\begin{matrix}
(1-\La)+\ep^\mu  \frac{\pd F}{\pd  I}+\cO(\ep^3)  &  \ep^\mu \sqrt{2\widetilde I_k}
(M^{q/p})'(\widetilde \phi^i) +\cO(\ep^3+\ep^{\mu+\beta})\\
\ep^2\, A+\ep^\mu \frac{\pd G}{\pd I}+\cO(\ep^{2+\beta}) & (1-\La) +\ep^\mu
\,\frac{\pd G}{\pd\phi} +\cO(\ep^3)
\end{matrix}\right)\\
&= (1-\La)^2+(1-\La)\bigg[\ep^\mu\,\bigg(\frac{\pd F}{\pd I}+ \frac{\pd G}{\pd
  \phi}\bigg)+\cO(\ep^3)\bigg] - \ep^{2+\mu} A\sqrt{2\widetilde I_k}(M^{q/p})'(\widetilde \phi^i)
  \\&+\cO(\ep^{2\mu}+\ep^5)\,,
\end{align*}
\endgroup
where all the functions are evaluated at $(\widetilde I_k,\widetilde \phi^i)$.
This is a quadratic equation for $1-\La$, which can be solved using that $(\widetilde I_k,\widetilde \phi^i)=(I_k,\phi^i)+\cO(\ep^\beta)$ to
obtain that the roots are
\[
1-\La_\pm=\pm2\ep^{\frac\mu2+1}\sqrt{A\sqrt{2I_k} (M^{q/p})'(\phi^i)}\,\big(1+\cO(\ep^{\beta}))+\cO(\ep^\mu)\,,
\]
where the terms represented as $\cO(\ep^j)$ are all real. Hence, for
small enough~$\ep$, the
roots $\La_\pm$ have nonzero imaginary part if $A$ and
$(M^{q/p})'(\phi^i)$ have distinct sign and are real numbers with
$\La_-<1<\La_+$ otherwise. Since the Poincar\'e map is area-preserving
and there are $j/2$ zeros for which $(M^{q/p})'(\phi)$ is positive and
$j/2$ for which it is negative, we immediately infer that half of the
$j$ periodic trajectories are elliptic and half of them are hyperbolic
of saddle type.
\end{proof}

Let us now apply this theorem to the Beltrami field $v$ after taking a normal datum
of the form
\[
f = \sum_{n \in \cN} f_n + \hat f\,,
\]
where the functions $f_n$ have been fixed in Subsection~\ref{S:KAM:case2} and Step 2, and we take
\[
\hat f = b(\al) \ep^{2+\mu} \cos\te \,.
\]
It is straightforward to see that the subharmonic Melnikov function
associated to this field is
\[
M^{q/p}(\phi) = \int_0^{q \ell} b (s)\cos(\phi+T(s)) ds\,,
\]
which is independent of the value $I_k$ of the approximate resonant torus. Notice that the functions
$b(s) \cos(T(s))$ and $b(s) \sin(T(s))$ are $q\ell$-periodic, since we
assumed the torsion $\tau$ to
satisfy the resonant condition~\eqref{eq:T0pq}. Therefore, it is clear that we can
choose the function $b(s)$ so that $M^{q/p}$ has two zeros that are non-degenerate. According to Theorem~\ref{T:subh:Melnikov}, the Beltrami field $v$ has a hyperbolic periodic orbit which is a cabling of the core curve $\ga$ near each approximate resonant torus $\{I=I_k\}$ for all $1\leq k\leq M$; in the coordinates $(\alpha,\theta)$ these orbits are $(q,p)$-torus knots, so they are all isotopic to each other.
The fact that we can construct a perturbation whose corresponding Melnikov
function does not depend on $I$ is crucial to create hyperbolic periodic orbits for all the $I_k$
simultaneously.

\subsubsection*{Step 4: Globalization of the local construction}

Since one can safely multiply the functions $\Ga_n$ by a
constant, taking $c\,\Ga_n$ instead, it is clear that one can assume
that $\|\Ga_3\|_{C^0}$ is arbitrarily small. Taking (for instance)
$\mu=5/2$, this implies that one can apply Theorem~\ref{T:KAM} to obtain the estimate of the measure of the set of ergodic invariant tori of the field $v$ in $\cT_\ep$.

Let us denote by
\[
K:=\bigcup_{j=1}^N \cT_\ep(\ga_j)
\]
the union of the thin tubes, and call $\tv$ the vector field on~$K$
that we define on each tube $\cT_\ep(\ga_j)$ for $\ep$ small enough as~$v_j \equiv v$ in the
notation that we have used in Steps~1--3. This is a local Beltrami
field, meaning that it satisfies the equation
\[
\curl \tv= \la\tv
\]
in~$K$, for some $\la=\cO(\ep^3)$. To globalize this local Beltrami field, we can use the
approximation theorem proved in~\cite[Theorem 8.3]{Acta}:

\begin{theorem}
Let $\tv$ be a vector field that satisfies the Beltrami equation
\[
\curl \tv = \lambda \tv
\]
in a bounded set $K \subset \RR^3$ with smooth boundary, where $\lambda$ is a nonzero constant
and the complement $\RR^3\backslash K$ is connected. Then there is a global
Beltrami field $u$, satisfying the equation
\[
\curl u = \lambda u
\]
in $\RR^3$, which falls off at infinity as $|u(x)| < \frac C{1+|x|}$ and approximates
the field in the $C^k$ norm as
\[
\|u-v\|_{C^k(K')} < \delta_0\,.
\]
Here $\delta_0$ is any fixed positive constant and $K'$ is any set whose
closure is contained in~$K$.
\end{theorem}

Since the error $\delta_0$ in the global approximation and the thickness
$\ep$ of the tubes are independent parameters, taking
$\delta_0=\cO(\ep^3)$, all the previous considerations apply to the global
Beltrami field $u$ inside each tube $\cT_\ep(\gamma_j)$, and the main
theorem follows.
\end{proof}

\section{Beltrami fields on the torus}
\label{S.torus}

In view of applications, our objective in this short section is to
state a result for Beltrami fields on the flat torus
$\TT^3:=(\RR/2\pi\ZZ)^3$ that is analogous to our main result (Theorem~\ref{T.main}).

Before we can formulate the theorem, let us recall that a Beltrami field on the flat torus $\TT^3$ is an eigenfunction of the curl operator:
\begin{equation}\label{BeltramiT}
\curl U= J\, U\,.
\end{equation}
It is well known that there only exist nontrivial Beltrami fields
of frequency~$J$ when this number is of the form $ J=\pm|k|$, where
$k\in\ZZ^3$ is a 3-vector of integer components. To put it
differently, the spectrum of the curl operator is the set of numbers
of this form. Furthermore, the most general Beltrami field of
frequency~$ J$ is a vector-valued trigonometric polynomial of the form
\[
U=\sum_{|k|=\pm J}\Big(b_k \, \cos(k\cdot x)+\frac{b_k\times
  k} J\,\sin(k\cdot x)\Big)\,,
\]
where $b_k\in\RR^3$ are vectors orthogonal to~$k$: $k\cdot
b_{k}=0$.

The main result of this section can then be stated as follows:

\begin{theorem}\label{T.torus}
Let $\cT_1, \ldots, \cT_N$ be pairwise disjoint (possibly knotted and
linked) contractible tubes in $\TT^3$.
For any large enough odd integer~$J$ and any~$\de>0$ there exists a diffeomorphism
$\Phi$ of $\TT^3$ such that $\Phi(\cT_1), \ldots,
\Phi(\cT_N)$ are invariant solid tori of a Beltrami field
satisfying the equation $\curl U = JU$ in $\TT^3$. Moreover, the
interior of each tube $\Phi(\cT_j)$ contains a set of
ergodic invariant tori of measure greater than $(1-\de)
|\Phi(\cT_j)|$, and at least $M$ hyperbolic periodic orbits that are
isotopic to each other and cablings of the core curve of~$\cT_j$.
\end{theorem}

Theorem~\ref{T.torus} is an immediate consequence of Theorem~\ref{T.Acta2} and of the inverse localization theorem
proved in~\cite[Theorem 2.1]{torus}:

\begin{theorem}\label{T.approx}
Let $u$ be a Beltrami field in $\RR^3$, satisfying
$\curl u=\la u$. Let us fix
any positive numbers $\ep'$ and $m$ and any bounded set $B\subset\RR^3$. Then for any large enough odd integer $J$ there is a
Beltrami field~$U$ on the torus, satisfying $\curl U=J U$ in $\TT^3$, such that
\begin{equation*}
\bigg\|U\bigg(\frac{\la}{ J}\;\cdot \bigg)-u\bigg\|_{C^m(B)}<\ep'\,.
\end{equation*}
\end{theorem}

\section*{Acknowledgments}

A.E.\ and D.P.-S.\ are respectively supported in part by the ERC
Starting Grants 633152 and 335079. A.L.\ is
supported by the Knut och Alice Wallenbergs stiftelse KAW
2015.0365. D.P.S.\ is supported in part by the grant MTM-2016-76702-P
of the Spanish Ministry of Science. This work is also partly supported by the ICMAT--Severo Ochoa grant SEV-2015-0554.

\appendix
\section{The Poincar\'e map of a harmonic field}

In this section we compute, perturbatively in the small
parameter~$\ep$, the Poincar\'e map of the unique harmonic field $h$
in the tube $\cT_\ep$ introduced in Subsection~\ref{sec:harmo}. This
analysis is important for two reasons. First, when the Beltrami
field $v$ in Subsection~\ref{ssec:lBF} is tangent to the boundary
(i.e. the normal datum is $f=0$), it has the form $v=h+\cO(\ep^3)$,
which is the case considered in~\cite{Acta}. The results in this
Appendix then recover the main theorem in~\cite{Acta}, and sharpen it
by showing the existence of an almost full measure set of invariant
tori in the interior of the vortex tube. Second, in the computations
below we shall assume that the field $h$ is perturbed by another
divergence-free field of size $\cO(\ep^\mu)$, with $\mu\in(2,3)$; this
shows why the subharmonic Melnikov method we used in the proof of the
main theorem does not work in this setting, which is the reason for
which we  introduced a normal component of order $\cO(\ep^3)$ on the boundary of the tube in Subsection~\ref{S.normal} to modify the frequency function of the Beltrami field $v$. All along this appendix, we shall use the coordinates $(\alpha,y)\in \SS^1_\ell\times\DD$ and the notation introduced in Subsection~\ref{S.Belt1}.

Let $\tilde h$ be a vector field in $\cT_\ep$ of the form
$$
\tilde h=h+Y
$$
where $Y$ is a smooth divergence-free vector field in the tube such that $Y=\cO(\ep^\mu)$, $\mu\in(2,3)$, and $h$ is the harmonic field of the tube. Following Subsection~\ref{S:Poinc}, we introduce the isochronous vector field
$$
X:=\frac{\tilde h}{\tilde h_\alpha}
$$
to compute the Poincar\'e map $\Pi$ of the field $\tilde h$ at the section $\{\alpha=0\}$. Notice that $\tilde h_\alpha=1+\cO(\ep)$, so $X$ is well defined, as well as the Poincar\'e map $\Pi:\DD_R\to\DD$, $R<1$, which is a diffeomorphism onto its image. The explicit expression of $X$ up to order $\cO(\ep^3)$ is obtained from Proposition~\ref{prop:BeltramiX} by setting $u_r^{(1)}=u_r^{(2)}=u_\theta^{(1)}=u_\theta^{(2)}=0$.

The first result of this appendix shows that $X$ (and hence $\tilde h$) has an approximate first integral, without any further assumptions on the curve $\gamma$. In the statement, we denote by $X(I)$ the action of the vector
field~$X$ (viewed as a first order differential operator) on the
scalar function~$I$ {(which should not be mistaken for the
action variable introduced in Section~\ref{S:KAM:case1})}.

\begin{proposition}\label{P.fi}
The smooth function defined on
$\SS^1_\ell\times\DD$ by
\[
I (\al,r,\theta):=
I^{(0)}(r) + \ep I^{(1)}(\al,r,\theta) +
\ep^2 I^{(2)}(\al,r,\theta)\,,
\]
with
\begin{align*}
I^{(0)} := {} & \frac{r^2}{2} \,, \\
I^{(1)} := {} & \frac{3(r^3-r) }{8} \kappa \cos \theta \,, \\
I^{(2)} := {} & 
- r \int_0^\al
{H_1(s,r) ds}
 + \frac{9(r^4-r^2)}{32} \kappa^2  + \frac{17(r^4-r^2)}{96} \kappa^2  \cos 2\theta\,,
\end{align*}
is an approximate first integral of the field~$X$ in the sense that
\[
X( I) = \cO(\ep^\mu)\,.
\]
If $Y=0$, then $h(I)=\cO(\ep^3)$.
\end{proposition}
\begin{proof}
Let us write
\[
X(I)=:a_0+\ep a_1+\ep^2 a_2+\cO(\ep^\mu)\,.
\]
Using the explicit expressions of $I^{(i)}$, $X_r^{(i)}$, and
$X_\theta^{(i)}$ and gathering the powers of $\ep$ we find
\begingroup
\allowdisplaybreaks
\begin{align*}
a_0 = {} & \frac{\pd I^{(0)}}{\pd \al} -  \tau \frac{\pd I^{(0)}}{\pd \te} =0\,,\\
a_1 = {} &
\frac{\pd I^{(1)}}{\pd \al}
+ \frac{\pd I^{(0)}}{\pd r}  X_r^{(1)}
- \tau \frac{\pd I^{(1)}}{\pd \te}
+ \frac{\pd I^{(0)}}{\pd \te}  X_\theta^{(1)} \\
           = {} & \frac{3(r^3-r)}{8} \kappa'  \cos\theta - r \left[ \frac{3r^2-3}{8}(\tau \kappa \sin\theta+\kappa' \cos\theta) \right]
                + \tau\frac{3 (r^3-r) }{8}\kappa \sin \theta  = 0\,, \\
a_2 = {} & \frac{\pd I^{(2)}}{\pd \al} + \frac{\pd I^{(1)}}{\pd r}  X_r^{(1)} + \frac{\pd I^{(0)}}{\pd r}  X_r^{(2)}
                 -\tau \frac{\pd I^{(2)}}{\pd \te} + \frac{\pd I^{(1)}}{\pd \te}  X_\theta^{(1)} + \frac{\pd I^{(0)}}{\pd\te}  X_\theta^{(2)} \\
           = {} & - r H_1 + \frac{9(r^4-r^2)}{16} \kappa \kappa' + \frac{17(r^4-r^2)}{48}\kappa \kappa'  \cos 2\theta \\
                & \qquad \qquad \qquad - \left[ \frac{3(3r^2-1)}{8} \kappa  \cos\theta \right]\cdot\left[ \frac{3r^2-3}{8}(\tau \kappa \sin \theta+\kappa' \cos\theta\right] \\
                & \qquad \qquad \qquad \qquad+ r \left[ \frac{r-r^3}{6}(\tau \kappa^2 \sin 2\theta + \kappa \kappa' \cos 2\theta) + \frac{3(r^3-r)}{8} \kappa \kappa'+H_1 \right] \\
                & \qquad \qquad \qquad \qquad \qquad \qquad+ \frac{17(r^4-r^2)}{48} \tau\kappa^2  \sin 2\theta \\
                & \qquad \qquad \qquad\qquad+ \left[ -\frac{3(r^3-r) }{8}\kappa \sin \theta \right]\cdot\left[ \frac{r^3-3}{8r}(-\tau \kappa \cos\theta+\kappa' \sin \theta)\right]=0
\end{align*}
\endgroup
To conclude, notice that if $Y=0$ then
$X(I)-a_0-\ep a_1-\ep^2 a_2$ is automatically of order $\cO(\ep^3)$,
proving the claim.
\end{proof}

We shall next explicitly compute the Poincar\'e map $\Pi(r_0,\theta_0)$ following the same reasoning as in Subsection~\ref{S:KAM:case2}, up to terms
of order $\cO(\ep^3)$. To state the result we use the function $T(s)$ defined in the aforementioned subsection.

\begin{proposition}\label{prop:Poin1}
In polar coordinates, the Poincar\'e map is
\begin{align*}
 \Pi_r(r_0,\theta_0) &= r_0 + \ep \Pi_r^{(1)}(r_0,\theta_0) + \ep^2 \Pi_r^{(2)}(r_0,\te_0) + \ep^\mu \Pi_r^{(\mu)}(r_0,\te_0) + \cO(\ep^3)\,, \\
 \Pi_\theta(r_0,\theta_0)&= \theta_0+T_0+\ep \Pi_\theta^{(1)}(r_0,\theta_0) + \ep^2 \Pi_\theta^{(2)}(r_0,\theta_0) + \ep^\mu \Pi_\te^{(\mu)}(r_0,\te_0) + \cO(\ep^3) \,,
\end{align*}
where
\begingroup
\allowdisplaybreaks
\begin{align*}
\Pi_r^{(1)}(r_0,\theta_0) &:= \frac{3-3r_0^2}{8} \kappa(0) \big[\cos (\theta_0+T_0)-\cos \theta_0 \big] \, , \\
\Pi_\theta^{(1)}(r_0,\theta_0)& :=\frac{r_0^2-3}{8r_0}\kappa(0)\big[\sin
                            (\theta_0+T_0)-\sin\theta_0 \big]\,,\\
\Pi_r^{(2)}(r_0,\theta_0) & := \frac{-55 r_0^4+46 r_0^2 -27}{768 r_0} \kappa(0)^2 (\cos 2(\theta_0+T_0)-\cos 2\te_0) \,\\
& \qquad \qquad \qquad +\frac{9r_0-3r_0^3}{32} \kappa(0)^2 \cos \te_0 (\cos (\te_0+T_0)-\cos \te_0) \\
& \qquad \qquad \qquad \qquad-\frac{3 r_0^4-12r_0^ 2+9}{64 r_0} \kappa(0)^2 \sin \te_0 (\sin (\te_0+T_0)-\sin \te_0)\,,\\
\Pi_\theta^{(2)}(r_0,\theta_0) &:= -\frac{12-5r_0^2}{32}\int_0^\ell
{\kappa (s)^2\,\tau(s)\, ds}
+\frac{3(r_0^4+2r_0^2-3)}{64r_0^2}\kappa(0)^2 \cos\theta_0 \sin (\theta_0+T_0)\\
& -\frac{(3-r_0^2)^2}{64r_0^2}\kappa(0)^2\sin\theta_0 \cos(\theta_0+T_0) +\frac{27-50r_0^2+25r_0^4}{384r_0^2}\kappa(0)^2\sin {2(\theta_0+T_0)}\\
& \qquad \qquad \qquad \qquad \qquad \qquad \qquad \qquad+\frac{27+14r_0^2-31r_0^4}{384r_0^2}\kappa(0)^2\sin {2\theta_0}\,,\\
\Pi_r^{(\mu)}(r_0,\te_0) &:= \int_0^\ell X_r^{(\mu)}(s,r_0,\te_0+T(s)) ds \,, \\
\Pi_\te^{(\mu)}(r_0,\te_0) &:= \int_0^\ell X_\te^{(\mu)} (s,r_0,\te_0+T(s))] ds \,.
\end{align*}
\endgroup
Moreover, $\Pi$ preserves the area measure on the disk given by
\begin{equation*}
B \tilde h_\al|_{\al=0}\, r \,dr \, d\theta = (1 + \ep \kappa(0) r \cos \theta
+ \cO(\ep^2))\, r \,dr \, d\theta\,.
\end{equation*}
\end{proposition}

\begin{proof}
As in Subsection~\ref{S:KAM:case2}, to obtain the Poincar\'e
map it suffices to compute the integral curves of the field~$X$ at
time~$\ell$. For this we will write the ODEs as integral
equations. Since $\al(s)=s$, this system reads simply as
\begin{align*}
r(s)   = {} & r_0
+ \sum_{j=1,2,\mu} \ep^j \int_0^s X_r^{(j)}(\sigma,r(\si),\te(\si))\,d\sigma
+ \cO(\ep^3), \\
\te(s) = {} & \te_0 +T(s)
+ \sum_{j=1,2,\mu} \ep^j \int_0^s X_\te^{(j)}(\sigma,r(\si),\te(\si))\, d\sigma
+ \cO(\ep^3),
\end{align*}
where we will assume throughout that $s$ takes values in $ [0,\ell]$.

To solve these equations
perturbatively, let us introduce the
notation
\begin{equation*}\label{eq:exp:rt}
\begin{split}
r(s)   &=:
r^{(0)}(s) + \ep r^{(1)} (s)
+\ep^2 r^{(2)}(s) + \ep^\mu r^{(\mu)} (s)+\cO(\ep^3) \, , \\
\te(s) &=:  \te^{(0)}(s) + \ep \te^{(1)}(s)
+ \ep^2 \te^{(2)}(s) +\ep^\mu \te^{(\mu)}(s)+
\cO(\ep^3) \, .
\end{split}
\end{equation*}
Using the explicit expressions of the zeroth-order terms one immediately
obtains that
\[
r^{(0)}(s) = r_0 \, , \qquad
\te^{(0)}(s) = \te_0 + T(s) \,.
\]

To compute the first-order terms, for notational convenience we will
denote by
\[
X^{(i)}_j[\si]:=X^{(i)}_j(\si,r_0,\te_0+T(\si))
\]
the component $X^{(i)}_j$ of the field~$X$ evaluated on the
unperturbed integral curve $(\si,r_0,\te_0+T(\si))$. A similar
notation will be used with the partial derivatives of the components of~$X$. With this
notation, the first-order terms can be readily shown to be
\begingroup
\allowdisplaybreaks
\begin{align*}
r^{(1)} (s)
= {} & \int_0^sX_r^{(1)}[\si] \, d\sigma \\
= {} &
-\frac{3 r_0^2 -3 }{8} \int_0^s \left(\tau(\sigma) \kappa (\sigma)\sin (\te_0+T(\sigma))+
\kappa'(\si) \cos (\te_0+T(\sigma)) \right) \, d\sigma \\
= {} &
\frac{3-3 r_0^2  }{8} \int_0^s \frac d{d\si}\Big(\kappa (\si)
\cos(\te_0+T(\si)) \Big)\, d\sigma \\
= {} & \frac{3-3r_0^2}{8} \left[ \kappa (s) \cos(\te_0+T(s))-\kappa(0) \cos
\te_0 \right] \, . \\
\te^{(1)} (s)
= {} & \int_0^sX_\te^{(1)}[\si] \, d\sigma \\
= {} &
\frac{r_0^2 -3 }{8 r_0}\int_0^s  \left(-\tau(\sigma) \kappa (\sigma)\cos
(\te_0+T(\sigma))+
\kappa'(\si) \sin (\te_0+T(\sigma)) \right) \, d\sigma \\
= {} & \frac{r_0^2-3}{8 r_0} \left[ \kappa (s) \sin(\te_0+T(s))-\kappa(0) \sin
\te_0 \right] \, .
\end{align*}
\endgroup

The computation of  $r^{(2)}(s)$ goes along the same lines:
\begingroup
\allowdisplaybreaks
\begin{align}
r^{(2)} (s)
= {} &
\int_0^s
\bigg(X_r^{(2)}[\si]
+\frac{\pd X_r^{(1)}}{\pd r} [\si]\, r^{(1)}(\sigma)
+\frac{\pd X_r^{(1)}}{\pd \te}[\si]\, \te^{(1)}(\sigma)\bigg) \, d\sigma \nonumber \\
= {} &
-\left[ \frac{r_0^3-r_0}{6} + \frac{9(r_0-r_0^3)}{64} + \frac{(3 r_0^2-3)(r_0^2-3)}{128 r_0} \right]
\int_0^s \nonumber \\
& \quad
[\tau(\sigma) \kappa^2(\sigma) \sin (2(\theta_0+T(\sigma))) \nonumber \\
&\qquad \qquad \qquad \qquad\qquad+
\kappa(\sigma) \kappa'(\sigma) \cos (2(\theta_0+T(\sigma)))]\, d\sigma \nonumber \\
& -\frac{9 r_0^3-9r_0}{32} \ka(0) \cos \theta_0
\int_0^s [\tau(\sigma) \kappa (\sigma) \sin (\theta_0+T(\sigma)) \nonumber \\
&\qquad \qquad \qquad \qquad \qquad \qquad \qquad \qquad \qquad+
\kappa'(\sigma) \cos(\theta_0+T(\sigma))] \, d \sigma \nonumber \\
& +\frac{-3r_0^4+12 r_0^2 - 9}{64 r_0} \ka(0) \sin \theta_0
\int_0^s [-\tau(\sigma) \kappa (\sigma) \cos (\theta_0+T(\sigma)) \nonumber \\
&\qquad \qquad \qquad \qquad \qquad \qquad \qquad \qquad \qquad+
\kappa'(\sigma) \sin(\theta_0+T(\sigma))] \, d \sigma \nonumber \\
& +
\left[ \frac{3(r_0^3-r_0)}{8} - \frac{9(r_0-r_0^3)}{64} + \frac{(3 r_0^2-3)(r_0^2-3)}{128 r_0}\right]
\int_0^s \kappa(\sigma) \kappa'(\sigma) \, d \sigma \nonumber \\
& \qquad \qquad \qquad \qquad \qquad \qquad \qquad \qquad \qquad+ \int_0^s H_1(\sigma,r_0)\,d\sigma\nonumber  \\
= {} &
-\frac{19 r_0^4-46 r_0^2 + 27}{770 r_0} \int_0^s\frac d{d\si}\Big(
\kappa(\sigma)^2 \cos (2(\theta_0+T(\sigma)))\Big)\, d\sigma
\label{eq:comp:theta2} \\
& -\frac{9r_0^3-9r_0}{32} \ka(0) \cos \theta_0
\int_0^s \frac d{d\si}[
\kappa(\sigma) \cos(\theta_0+T(\sigma))] \, d \sigma \nonumber \\
& +\frac{-3r_0^4+12 r_0^2 - 9}{64 r_0} \ka(0) \sin \theta_0
\int_0^s \frac d{d\si}[
\kappa(\sigma) \sin(\theta_0+T(\sigma))] \, d \sigma \nonumber \\
& +\frac{69 r_0^4-78 r_0^2 +9}{256 r_0}
\int_0^s \frac d{d\si}\big(\kappa(\sigma)^2 \big)\, d \sigma
+ \int_0^s H_1(\sigma,r_0)\,d\sigma \,.\nonumber
\end{align}
\endgroup
Integrating the derivatives with respect to~$\si$ and using that
$H_1$ has zero mean (that is, $\int_0^\ell H_1(\si,r)\, d\si=0$)
we readily arrive
at the formula for $\Pi_r^{(2)}$ presented in the statement.

Likewise,
\begingroup
\allowdisplaybreaks
\begin{align*}
\te^{(2)} (s)
= {} &
\int_0^s
\bigg(X_\te^{(2)}[\si]
+\frac{\pd X_\te^{(1)}}{\pd r} [\si] r^{(1)}(\sigma)
+\frac{\pd X_\te^{(1)}}{\pd \te}[\si] \te^{(1)}(\sigma)\bigg) \,d\sigma  \\
= {} &
\left[ \frac{7 r_0^2-8}{48} + \frac{(r_0^2+3)(3-3r_0^2)}{128 r_0^2} + \frac{(r_0^2-3)^2}{128 r_0^2}\right]
\int_0^s
[-\tau(\sigma) \kappa^2(\sigma) \cos (2(\theta_0+T(\sigma)))\\
&\qquad \qquad \qquad \qquad \qquad \qquad \qquad+
\kappa(\sigma) \kappa'(\sigma) \sin (2(\theta_0+T(\sigma)))] \, d\sigma
\\
& + \frac{3 r_0^4+6r_0^2 -9}{64 r_0^2} \kappa(0) \cos\theta_0
\int_0^s [-\tau(\sigma) \kappa (\sigma) \cos (\theta_0+T(\sigma))\\
&\qquad \qquad \qquad \qquad \qquad \qquad \qquad \qquad+
\kappa'(\sigma) \sin(\theta_0+T(\sigma))]  \, d \sigma \\
& - \frac{r_0^4 -6r_0^2+9}{64 r_0^2} \kappa(0) \sin \te_0
\int_0^s [\tau(\sigma) \kappa (\sigma) \sin (\theta_0+T(\sigma))\\
&\qquad \qquad +
\kappa'(\sigma) \cos(\theta_0+T(\sigma))] \, d \sigma \\
&
-
\left[ \frac{3-r_0^2}{8} + \frac{(r_0^2+3)(3-3r_0^2)}{128 r_0^2} - \frac{(r_0^2-3)^2}{128 r_0^2}\right]
\int_0^s \tau(\sigma) \kappa(\sigma)^2 \, d
  \sigma\\
= {} &
\frac{25 r_0^4 - 50 r_0^2 + 27}{484 r_0^2 }
\int_0^s\frac d{d\si}\big(
\kappa(\sigma)^2 \sin (2(\theta_0+T(\sigma)))\big) \, d\sigma
\\
& + \frac{3 r_0^4+6r_0^2 -9}{64 r_0^2} \kappa(0) \cos\theta_0
\int_0^s \frac d{d\si}[
\kappa(\sigma) \sin(\theta_0+T(\sigma))]  \, d \sigma \\
& - \frac{r_0^4 -6r_0^2+9}{64 r_0^2} \kappa(0) \sin \te_0
\int_0^s \frac d{d\si}[
\kappa(\sigma) \cos(\theta_0+T(\sigma))] \, d \sigma \\
& - \frac{12-5r_0^2}{32} \int_0^s \tau(\sigma) \kappa(\sigma)^2 \, d \sigma\,.
\end{align*}
\endgroup
Carrying out the integrations we then arrive at the expression for $\Pi_\te^{(2)}$
presented in the statement.

The expression for $\Pi^{(\mu)}_i$ is obvious because
\[
r^{(\mu)}(s)=\int_0^s X^{(\mu)}_r[\si]\, d\si\,,\qquad \te^{(\mu)}(s)=\int_0^s X^{(\mu)}_\te[\si]\, d\si\,.
\]
Since $\tilde h$ is divergence-free (so its flow preserves the volume $B\,
r\, d\al\, dr\,d\te$) and $X:=\tilde h/\tilde h_\al$, the same proof as in
\cite[Proposition 7.3]{Acta} shows that the Poincar\'e map preserves
the area measure $B\tilde h_\al|_{\al=0}\, r\, dr\,d\te$.
\end{proof}

It is useful to have the expression of the Poincar\'e map not only in
polar coordinates $(r_0,\te_0)$, as above, but also in the coordinates
$(I,\te_0)$. Specifically, one
can parametrize the points of the disk $\DD$ by the angle~$\te_0$ and the
approximate first integral~$I$, and the Poincar\'e map then reads as follows. Notice that the variable~$I$ ranges in an interval of the form
\begin{equation*}
\cI_\ep:=\bigg\{I:O(\ep)<I<\frac12+\cO(\ep)\bigg\}\,.
\end{equation*}
Since there is no risk of confusion, we shall write $r\equiv r_0$ and $\theta\equiv \theta_0$.
\begin{corollary}\label{C.Poincare}
The Poincar\'e map $\Pi$ reads in the coordinates $(I,\te)$
as
\begin{equation}\label{eq:Poincare:I}
\begin{split}
 \Pi_I(I,\theta) = {} & I + \ep^\mu \widehat\Pi_I^{(\mu)} (I,\te) + \cO(\ep^3) \, ,\\
 \Pi_\theta (I,\theta)= {} &\theta+T_0+
\ep  \widehat\Pi_\theta^{(1)}(I,\theta) + \ep^2  \widehat\Pi_\theta^{(2)} (I,\theta)
+ \ep^\mu \widehat\Pi_\te^{(\mu)} (I,\te)
+ \cO(\ep^3) \, ,
\end{split}
\end{equation}
where
the functions $ \widehat\Pi_\theta^{(1)}$, $ \widehat\Pi_\theta^{(2)}$,
$ \widehat\Pi_I^{(\mu)}$ and $ \widehat\Pi_\theta^{(\mu)}$ are given by
\begingroup
\allowdisplaybreaks
\begin{align*}
 \widehat\Pi_\theta^{(1)} (I,\theta) &:= \frac{2I-3}{8 \sqrt{2I}} \kappa(0) \big[ \sin (\theta+T_0) -\sin \theta \big] \, ,\\
 \widehat\Pi_\theta^{(2)} (I,\theta) &:=-\frac{6-5 I}{16}\int_0^\ell \kappa (\al)^2\,\tau(\al)\, d\al
+\frac{(27-100 I+100 I^2)}{768 I}\kappa(0)^2\sin {2(\theta+T_0)}
\\
&- \frac{4I^2-12I+9}{128 I}\kappa(0)^2\sin\theta \cos(\theta+T_0)
 -\frac{11I-8}{96}\kappa(0)^2\sin {2\theta}\,,\\
\widehat \Pi_I^{(\mu)}(I,\te) &:= \sqrt{2I} \int_0^\ell \tilde h_r^{(\mu)}(\sigma,\sqrt{2I},\te+T(\sigma)) \, d\sigma \,, \\
\widehat\Pi_\te^{(\mu)}(I,\te) &:= \int_0^\ell [\tilde h_\te^{(\mu)}(\sigma,\sqrt{2I},\te+T(\sigma)) + \tau \tilde h_\al^{(\mu)} (\sigma,\sqrt{2I},\te+T(\sigma))] \, d\sigma \,.
\end{align*}
\endgroup
\end{corollary}
\begin{proof}
Proposition~\ref{prop:Poin1} obviously implies that if we plug the
approximate first integral of Proposition~\ref{P.fi}, then
\begin{equation}\label{eq:integral_map}
I(0,\Pi_r(r,\theta),\Pi_\theta(r,\theta))=I(0,r,\theta)+\cO(\ep^\mu)\,.
\end{equation}

It will be useful to express things in terms of the approximate first
integral $I$ and the angle $\theta$. To this end, notice that if $\ep$ is small enough, we can invert the expression of the first
integral in order to write $r$ as
\begin{equation}\label{eq:r_depend_I}
r= \sqrt{2I} -\frac{3(2I-1)}{8} \ep \kappa(0) \cos \theta + \cO(\ep^2)\,.
\end{equation}

Using \eqref{eq:integral_map} and \eqref{eq:r_depend_I}, it turns out
that the Poincar\'e map $ \Pi$ can be written in the variables
$(I,\theta)$ as
\begin{equation}\label{eq:Poincare:I}
\begin{split}
 \Pi_I(I,\theta) = {} & I + \ep^\mu  \Pi_I^{(\mu)} (I,\theta) + \cO(\ep^3)\, ,\\
 \Pi_\theta (I,\theta)= {} &\theta+T_0+
\ep  \widehat\Pi_\theta^{(1)}(I,\theta) + \ep^2  \widehat\Pi_\theta^{(2)} (I,\theta) +\ep^\mu  \widehat\Pi_\theta^{(\mu)} (I,\theta) + \cO(\ep^3) \, .
\end{split}
\end{equation}
Here the function
\[
\widehat\Pi_\theta^{(1)}(I,\te) := \frac{2I-3}{8 \sqrt{2I}} \kappa(0) \big[
\sin (\theta+T_0) -\sin \theta \big]
\]
is just the function
$\Pi_\te^{(1)}(r,\te)$ defined in Proposition~\ref{prop:Poin1} and
expressed in terms of the variables $(I,\te)$ (in this change of
variables, one must also evaluate at $\ep=0$ to ensure that it does
not depend on~$\ep$).  In turn, the function
\begin{gather*}
 \widehat\Pi_\theta^{(2)} (I,\theta) :=-\frac{6-5 I}{16}\int_0^\ell \kappa (\al)^2\,\tau(\al)\, d\al
+\frac{(27-100 I+100 I^2)}{768 I}\kappa(0)^2\sin {2(\theta+T_0)}
\\
- \frac{4I^2-12I+9}{128 I}\kappa(0)^2\sin\theta \cos(\theta+T_0)
 -\frac{11I-8}{96}\kappa(0)^2\sin {2\theta}
\end{gather*}
is calculated using both the expression of the function
$\Pi_\theta^{(2)} (r,\theta)$ in terms of the variables $(I,\te)$
(again setting $\ep=0$ in the change of variables) and first order
terms coming from the expression of $\Pi_\te^{(1)}(r,\te)$ in these variables.

To conclude, we observe that the terms of order $\mu$ do not get any
additional contribution coming from lower order terms, so
$\widehat\Pi_\te^{(\mu)}(I,\te)$ and $\widehat\Pi_I^{(\mu)}(I,\te)$ are obtained by evaluating the corresponding term of
order~$\mu$ of the Poincar\'e map in polar coordinates at $r=\sqrt{2I}$. In the formulas of the statement we have used that
$X_r^{(\mu)}= \tilde h_r^{(\mu)}$ and $X_\theta^{(\mu)} = \tilde h_\theta^{(\mu)}+\tau \tilde h_\al^{(\mu)}$ (the same computation as in Proposition~\ref{prop:BeltramiX}).
\end{proof}

It is remarkable that without any further assumptions on the curvature and torsion (compare with the resonant hypothesis~\eqref{eq:T0pq}), the Poincare map of the field $\tilde h$ can be written as an $\cO(\ep^\mu)$ perturbation of an integrable twist map, analogous to the expression of the iterated Poincar\'e map of the Beltrami field $v$ in Proposition~\ref{eq:Prop:Poin:res}. This is achieved by finding an appropriate angle variable.

\begin{proposition}\label{prop:freq}
The Poincar\'e map has the approximate frequency function
\begin{equation}\label{eq:trunc:freq2}
 \omega(I) := T_0 - \ep^2 \frac{6-5 I}{16} \int_0^\ell \kappa(\al)^2
\tau(\al) \, d\al \,,
\end{equation}
in the sense that there is a new $\SS^1$-valued angular variable $\phi\equiv\phi(I,\te)=\te+\cO(\ep)$
on the punctured disk
such that the Poincar\'e map in the variables $(I,\phi)$ is in normal form modulo a small error:
\begin{equation}\label{Pican}
\begin{split}
 \Pi_I(I,\phi) = {} & I + \ep^\mu \Pi_I^{(\mu)} (I,\phi) + \cO(\ep^3) \, ,\\
 \Pi_\phi (I,\phi)= {} &\phi+\omega(I)
+ \ep^\mu \Pi_\phi^{(\mu)} (I,\phi)
+ \cO(\ep^3) \, .
\end{split}
\end{equation}
Here
\begin{align*}
\Pi_I^{(\mu)}(I,\phi) &:= \sqrt{2I} \int_0^\ell \tilde h_r^{(\mu)}(\sigma,\sqrt{2I},\phi+T(\sigma)) \, d\sigma \,, \\
\Pi_\phi^{(\mu)}(I,\phi)&:= \int_0^\ell
                                 [\tilde h_\te^{(\mu)}(\sigma,\sqrt{2I},\phi+T(\sigma))
                                 + \tau \tilde h_\al^{(\mu)}
                                 (\sigma,\sqrt{2I},\phi+T(\sigma))] \,
                                 d\sigma \,.
\end{align*}
\end{proposition}

\begin{proof}
We will apply the Poincar\'e--Lindstedt method with a new angular
variable of the form
\[
\phi(\te,I):=
\theta + \ep \phi^{(1)}(I,\theta) + \ep^2 \phi^{(2)}(I,\theta)\,.
\]
By Equation~\eqref{eq:Poincare:I}, this automatically implies that
\[
\Pi_I(I,\phi)=I+\cO(\ep^\mu)
\]
in these coordinates $(I,\phi)$ (and $\Pi_I(\te,I) =I+\cO(\ep^3)$ if $X^{(\mu)}=0$), so we only need to worry about the
angular component of the Poincar\'e map.

Imposing $\Pi_\phi(I,\phi)= \phi+\om(I) +\cO(\ep^\mu)$ is tantamount
to demanding that
\begin{equation}\label{condphiom}
\phi+\om(I)=\Pi_\te(I,\phi + \om(I))+\cO(\ep^\mu)\,.
\end{equation}
To compute $\phi$ perturbatively, let us take the ansatz
\[
\omega(I) = \omega^{(0)}(I) + \ep \omega^{(1)}(I) + \ep^2 \omega^{(2)}(I) \,,
\]
where the $I$-dependent functions $\om^{(j)}(I)$ are to be chosen later.
Since
\begin{align}
 \Pi_\theta(I,\phi) = {} & \theta + T_0 + \ep \phi^{(1)}(I,\theta) +
 \ep^2 \phi^{(2)}(I,\theta) + \ep  \Pi_\theta^{(1)}(I,\theta+\ep
 \phi^{(1)}(I,\theta)) \notag\\
&\qquad \qquad \qquad \qquad \qquad \qquad \qquad \qquad\qquad+\ep^2
\widehat \Pi_\theta^{(2)}(I,\theta) + \cO(\ep^\mu)\label{Pimu}
\end{align}
and
\begin{align*}
\phi(\theta+\omega(I),I) = {} & \theta + \omega^{(0)}(I) + \ep
\omega^{(1)}(I) + \ep^2 \omega^{(2)}(I) + \ep \phi^{(1)}(I,\theta+\ep
\omega^{(0)}(I)) \notag\\
&\qquad \qquad \qquad \qquad \qquad \qquad \qquad \qquad \qquad+ \ep^2 \phi^{(2)}(I,\theta) + \cO(\ep^3)\,,
\end{align*}
one can expand Equation~\eqref{condphiom} in~$\ep$ and identify terms to obtain
\[
\omega^{(0)}=T_0
\]
and the following homological equations for $\phi^{(1)}(I,\theta)$ and $\phi^{(2)}(I,\theta)$:
\begin{align}
\phi^{(1)}(I,\theta+{T_0})-\phi^{(1)}(I,\theta) = {} &
\Pi_\theta^{(1)}(I,\theta)-\omega^{(1)}(I)\, , \label{eq:Coho1}\\
\phi^{(2)}(I,\theta+{T_0})-\phi^{(2)}(I,\theta) = {} & \frac{\pd
\Pi_\theta^{(1)}(I,\theta)}{\pd \theta} \phi^{(1)}(I,\theta) + \widehat
\Pi_\theta^{(2)}(I,\theta) \notag\\
&\qquad\qquad- \frac{\pd \phi^{(1)} (I,\theta)}{\pd \theta}
\omega^{(1)}(I) - \omega^{(2)}(I)\,. \label{eq:Coho2}
\end{align}

To solve the homological equations one needs to choose
$\om^{(j)}(I)$ so that the right hand side has zero average, as this
is a necessary condition for the existence of
solutions. In particular, in Equation~\eqref{eq:Coho1} we readily see
that
\[
\int_0^{2\pi}\Pi_\theta^{(1)}(I,\theta)
\, d\theta = 0\,,
\]
so we infer that
\[
\omega^{(1)}(I) = 0\,.
\]
Equation \eqref{eq:Coho1} then becomes
\[
\phi^{(1)}(I,\theta+{T_0})-\phi^{(1)}(I,\theta) = \frac{2I-3}{8 \sqrt{2I}} \kappa(0)
[\sin (\theta+{T_0}) -\sin \theta] \,,
\]
so it is clear that the solution is given by
\begin{equation}\label{eq:sol:Coho1}
\phi^{(1)}(I,\theta)=\frac{2I-3}{8 \sqrt{2I}} \kappa(0)\sin \theta
\end{equation}
up to an inessential $I$-dependent constant that we can
take to be zero.

To solve \eqref{eq:Coho2}, we compute the average of the right hand
side of the equation to derive that
\begin{align*}
\omega^{(2)}(I) &= \frac{1}{2\pi} \int_0^{2 \pi}
\left( \frac{\pd
\Pi_\theta^{(1)}(I,\theta)}{\pd \theta} \phi^{(1)}(I,\theta) + \widehat
\Pi_\theta^{(2)}(I,\theta) \right) \, d\te \\
&
=
-\frac{6-5 I}{16} \int_0^l
\kappa(\al)^2 \tau(\al) \, d\al \,.
\end{align*}
The function $\phi^{(2)}$ is then readily shown to be
\[
\phi^{(2)}(I,\theta) =\frac{27-100 I + 100 I^2}{768 I}
\kappa(0)^2\sin 2\te
\]
up to an $I$-dependent constant. The theorem follows upon realizing
that, as $\phi=\te+\cO(\ep)$, the $\cO(\ep^\mu)$ terms in the
Poincar\'e map are given by the same expression as in
Corollary~\ref{C.Poincare} after substituting $\te$ by~$\phi$.
\end{proof}

We can now combine Proposition~\ref{prop:freq} with KAM theory to prove that for a generic
curve~$\ga$, the field~$\tilde h$ has a set of ergodic invariant tori
contained in $\cT_\ep$ of almost full measure and given by small
deformations of level sets of the approximate first
integral~$I$. Specifically, using the measure $d\al\, dy$ on
$\SS^1_\ell\times\DD$ so that $|\SS^1_\ell\times\DD|=\pi\ell$, one can state this result as follows:

\begin{proposition}\label{P.tori}
Suppose that
\[
\int_0^\ell \kappa(\al)^2 \tau(\al) \, d\al\neq0\,.
\]
Then for small enough~$\ep$ the field~$\tilde h$ has a set of
invariant tori of the form
$$
\{I + \cO(\ep^\mu)=\text{constant} \}\,,
$$
contained in $\SS^1_\ell\times\DD$,
whose measure is at least $\pi\ell -C\ep^{\frac \mu2-1}$.
\end{proposition}

\begin{proof}
By Proposition~\ref{prop:Poin1}, the Poincar\'e map of $X$ (and hence of $\tilde h$) at the section $\{\alpha=0\}$ preserves an
area measure of the form
$(1+\cO(\ep))\, dy$.
By Proposition~\ref{prop:freq}, there are polar coordinates $(I,\phi)$ on
the disk $\DD$ such that the Poincar\'e map is given by
\[
(\Pi_I(I,\phi),\Pi_\phi(I,\phi))=(I,\phi+\om(I))+\cO(\ep^\mu)\,.
\]
Moser's twist condition then reads
\[
\om'(I)=\frac{5\ep^2}{16} \int_0^\ell \kappa(\al)^2
\tau(\al) \, d\al \neq0\,.
\]
If the above
integral is nonzero, then the twist is a nonzero constant
of order $\ep^2$. Since the Poincar\'e map is an order $\cO(\ep^\mu)$
perturbation of an integrable twist map, it follows (see
e.g.~\cite{Siegel}) that for small enough $\ep$ all the disk $\DD$ but a
set of measure at most $C\ep^{\frac\mu2-1}$ is covered by quasi-periodic invariant
curves of the Poincar\'e map. In terms of the vector field~$\tilde h$, this obviously
means that the whole $\es\times\DD$ but a set of measure at most
$C\ep^{\frac\mu2-1}$ is covered by ergodic invariant tori of~$\tilde h$, as claimed.
\end{proof}

It is important to stress why we cannot effectively apply the Melnikov
method to the perturbation $\tilde h$ of the harmonic field, as this is why we had to introduce in Section~\ref{S.trajectories} a
carefully chosen $\cO(\ep^3)$~boundary datum to modify the frequency
function of the Beltrami field. The key difficulty is that one has a weak twist condition that one cannot effectively use with the
usual Melnikov theory. Roughly speaking, the reason is that the Melnikov
function involves integrals over a time interval of length of
order~$q$, where $p/q$ is the frequency of an approximate resonant
invariant torus. While in the classical setting the number~$q$ is
fixed, under a weak twist condition this number tends to infinity as
the perturbation parameter tends to zero with a rate that depends on
the arithmetic properties of the frequency function. 
{To put it differently,
the numbers $p$ and~$q$ tend to infinity as $\ep\to0$, and this has
the effect that the integral that one would like to use to define the
Melnikov function is so degenerate that it actually lacks a leading order term to
which one can apply the basic ideas of the Melnikov theory.}

\bibliographystyle{amsplain}

\end{document}